\documentclass{amsart}

\usepackage{amsthm,amsfonts,amsmath,amssymb,mathrsfs,verbatim,cite,youngtab}

\numberwithin{equation}{section}
\newtheorem{theorem}{Theorem}[section]
\newtheorem{corollary}[theorem]{Corollary}
\newtheorem{proposition}[theorem]{Proposition}

\newtheorem{lemma}[theorem]{Lemma}

\renewcommand{\and}{\quad\text{and}\quad}

\newcommand{\f}{\mathsf{f}}
\newcommand{\Rat}{\mathbb Q}
\newcommand{\Nat}{\mathbb N}
\newcommand{\Z}{\mathbb Z}
\newcommand{\F}{\mathbb F}

\newcommand{\Symm}{\mathfrak{S}}
\newcommand{\abs}[1]{\lvert#1\rvert}
\newcommand{\la}{\lambda}
\newcommand{\qbin}[2]{\genfrac{[}{]}{0pt}{}{#1}{#2}}
\newcommand{\qbinE}[2]{\genfrac{\langle}{\rangle}{0pt}{}{#1}{#2}}
\renewcommand{\P}{\mathsf{P}}
\newcommand{\D}{\mathcal{D}}
\newcommand{\E}{\mathcal{E}}
\newcommand{\Q}{\mathsf{Q}}
\newcommand{\R}{\mathsf{R}}

\newcommand{\gn}{\mathfrak{gl}_n}
\newcommand{\sn}{\mathfrak{sl}_n}
\newcommand{\spec}[1]{\langle #1\rangle}

\begin{document}

\title[Branching rules]{Branching rules for symmetric Macdonald polynomials 
and $\sn$ basic hypergeometric series}

\author{Alain Lascoux and S. Ole Warnaar}
\thanks{Work supported by the ANR project MARS (BLAN06-2 134516)
and the Australian Research Council}

\address{CNRS, Institut Gaspard Monge, Universit\'{e} Paris-Est,
Marne-La-Vall\'ee, France}
\address{School of Mathematics and Physics, 
The University of Queensland, Brisbane, QLD 4072, Australia}

\subjclass[2000]{05E05, 33D52, 33D67}

\begin{abstract}
A one-parameter generalisation $R_{\la}(X;b)$
of the symmetric Macdonald polynomials and interpolations Macdonald 
polynomials is studied from the point of view of branching rules.
We establish a Pieri formula, evaluation symmetry, principal
specialisation formula and $q$-difference equation for $R_{\la}(X;b)$.
We also prove a new multiple $q$-Gauss summation formula and 
several further results for $\sn$ basic hypergeometric series based
on $R_{\la}(X;b)$.
\end{abstract}

\maketitle

\section{Introduction}

Let $\la$ be a partition, i.e., 
$\la=(\la_1,\la_2,\dots)$ is a weakly decreasing sequence of
nonnegative integers such that $\abs{\la}:=\la_1+\la_2+\cdots$ is finite.
Let the length $l(\la)$ of $\la$ be the number of nonzero $\la_i$. 
For $x=(x_1,\dots,x_n)$ and $l(\la)\leq n$ the Schur function 
$s_{\la}(x)$ is defined as 
\begin{equation}\label{alt}
s_{\la}(x):=\frac{\det_{1\leq i,j\leq n} (x_i^{\la_j+n-j})}
{\det_{1\leq i,j\leq n} (x_i^{n-j})}
=\frac{\det_{1\leq i,j\leq n} (x_i^{\la_j+n-j})}{\Delta(x)},
\end{equation}
where $\Delta(x):=\prod_{i<j}(x_i-x_j)$ is the Vandermonde product.
If $l(\la)>n$ then $s_{\la}(x):=0$.
From its definition it is clear that $s_{\la}(x)$ is a 
symmetric polynomial in $x$ of homogeneous degree $\abs{\la}$, and that
$\{s_{\la}|~l(\la)\leq n\}$ forms a basis of the ring of 
symmetric functions $\Lambda_n:=\Z[x_1,\dots,x_n]^{\Symm_n}$.

A classical result for Schur functions is the combinatorial formula
\begin{equation}\label{Schur}
s_{\la}(x)=\sum_T x^T.
\end{equation}
Here the sum is over all semi-standard Young tableau $T$ of shape $\la$,
and $x^T$ is shorthand for the monomial $x_1^{\mu_1}x_2^{\mu_2}\cdots
x_n^{\mu_n}$ with $\mu_i$ the number of squares of the tableau 
filled with the number $i$. 
One of the remarkable facts of \eqref{Schur} is that it actually yields 
a symmetric function.

The conventional way to view a semi-standard Young tableau of
shape $\la$ (and length at most $n$) as a filling
of a Young diagram with the numbers $1,2,\dots,n$ such that squares 
are strictly increasing along columns and weakly increasing along
rows. Given two partitions (or Young diagrams) $\la,\mu$ write
$\mu\preccurlyeq\la$ if $\mu\subseteq\la$ and $\la-\mu$ is
a horizontal strip, i.e., if the skew diagram $\la-\mu$ contains
at most one square in each column.
Then an alternative viewpoint is to consider a Young tableau of shape
$\la$ as a sequence of partitions
\begin{equation}\label{Tpart}
0=\la^{(0)}\preccurlyeq
\la^{(1)}\preccurlyeq\dots\preccurlyeq\la^{(n)}=\la,
\end{equation}
where $0$ denotes the empty partition.
For example, for $n=6$ the tableau

\[
\young(11122246,224555,456,56,6)
\]
may be encoded as
\[
0\preccurlyeq(3)\preccurlyeq(6,2)\preccurlyeq(6,2)\preccurlyeq 
(7,3,1)\preccurlyeq(7,6,2,1)\preccurlyeq (8,6,3,2,1).
\]
The above description implies that a recursive
formulation of the Schur functions, equivalent to the combinatorial
formula \eqref{Schur},
is given by the \textit{branching rule}
\begin{equation}\label{bSchur}
s_{\la}(x_1,\dots,x_n)=
\sum_{\mu\preccurlyeq\la} 
x_n^{\abs{\la-\mu}} s_{\mu}(x_1,\dots,x_{n-1}),
\end{equation}
subject to the initial condition 
$s_{\la}(\text{--})=\delta_{\la,0}$.

If we let $\mu\subseteq\la$ be a pair of partitions and define
the skew Schur function $s_{\la/\mu}$ of a single variable $z$ as
\[
s_{\la/\mu}(z):=\begin{cases}
z^{\abs{\la-\mu}} &\text{if $\mu\preccurlyeq\la$} \\
0 &\text{otherwise,}
\end{cases}
\]
then the branching rule for Schur functions takes the more familiar form
\begin{equation}\label{bSchur2}
s_{\la}(x_1,\dots,x_n)=
\sum_{\mu\subseteq\la} s_{\la/\mu}(x_n) s_{\mu}(x_1,\dots,x_{n-1}).
\end{equation}

\medskip

The Macdonald polynomials $P_{\la}(x)=P_{\la}(x;q,t)$ 
\cite{Macdonald88,Macdonald95}
are an important $q,t$-generalisation of the Schur functions,
and the $P_{\la}$ for $l(\la)\leq n$
form a basis of the ring $\Lambda_{n,\F}:=\Lambda_n\otimes \F$,
where $\F=\Rat(q,t)$.
A classical result in the theory is that the Macdonald polynomials
satisfy a combinatorial formula not unlike that of the Schur functions;
\begin{equation*}
P_{\la}(x)=\sum_T \psi_T \, x^T,
\end{equation*}
where $\psi_T=\psi_T(q,t)\in\F$ is a function that admits an explicit
combinatorial description. Importantly, if $T$ has no more than
$n$ rows it factorises as
\[
\psi_T=\prod_{i=1}^n \psi_{\la^{(i)}/\la^{(i-1)}},
\]
where, as before, $0=\la^{(0)}\preccurlyeq \dots \preccurlyeq
\la^{(n)}=\la$ is the sequence of partitions representing $T$. 
Probably the simplest (albeit non-combinatorial) expression for 
$\psi_{\la/\mu}$ is \cite[Page 342]{Macdonald95}
\begin{equation}\label{psi}
\psi_{\la/\mu}=\prod_{1\leq i\leq j\leq l(\mu)}
\frac{f(q^{\mu_i-\mu_j}t^{j-i}) f(q^{\la_i-\la_{j+1}}t^{j-i})}
{f(q^{\la_i-\mu_j}t^{j-i}) f(q^{\mu_i-\la_{j+1}}t^{j-i})},
\end{equation}
where $f(a)=(at)_{\infty}/(aq)_{\infty}$ with $(a)_{\infty}=
\prod_{i\geq 0} (1-aq^i)$. (Note that $\psi_{\la/\mu}\in\F$ since
$\mu\preccurlyeq\la$.)
It follows from the above that the Macdonald
polynomials, like the Schur functions, can be described by a 
simple branching rule. Namely,
\begin{equation}\label{bP}
P_{\la}(x_1,\dots,x_n)=
\sum_{\mu\preccurlyeq\la} 
x_n^{\abs{\la-\mu}} \psi_{\la/\mu}\, P_{\mu}(x_1,\dots,x_{n-1}),
\end{equation}
subject to the initial condition $P_{\la}(\text{--})=\delta_{\la,0}$.
Again we may define a single-variable skew
polynomial $P_{\la/\mu}(z)=P_{\la/\mu}(z;q,t)$ for
$\mu\subseteq\la$
\begin{equation}\label{Plamua}
P_{\la/\mu}(z):=\begin{cases}
z^{\abs{\la-\mu}}\psi_{\la/\mu} &\text{if $\mu\preccurlyeq\la$} \\
0 &\text{otherwise}
\end{cases}
\end{equation}
to turn the branching formula for the Macdonald polynomials into
\begin{equation}\label{bP2}
P_{\la}(x_1,\dots,x_n)=
\sum_{\mu\subseteq\la} P_{\la/\mu}(x_n) P_{\mu}(x_1,\dots,x_{n-1}).
\end{equation}

In view of the above two examples of symmetric functions admitting a 
recursive description in the form of a branching formula,
a natural question is
\begin{quote}
\textit{Can one find more general branching-type formulas that 
lead to symmetric functions?}
\end{quote}
To fully appreciate the question we should point out that it is not at
all obvious that if one were to take \eqref{bSchur} (or, equivalently,
\eqref{bSchur2}) as the definition of the Schur functions or 
\eqref{bP} (or \eqref{bP2}) as the definition of the Macdonald
polynomials, that the polynomials in question are symmetric in $x$.

Assuming throughout that $\abs{q}<1$ let the (generalised) $q$-shifted 
factorials be defined as follows:
\begin{subequations}
\begin{align}
(b)_{\infty}=(b;q)_{\infty}&:=\prod_{i=0}^{\infty}(1-bq^i), \\
(b)_k=(b;q)_k&:=\frac{(b)_{\infty}}{(bq^k)_{\infty}}, \\
(b)_{\la}=(b;q,t)_{\la}&:=\prod_{i=1}^{l(\la)} (bt^{1-i})_{\la_i},
\label{bla}
\end{align}
\end{subequations}
and let $(b_1,\dots,b_i)_k=(b_1)_k\cdots (b_i)_k$ and
$(b_1,\dots,b_i)_{\la}=(b_1)_{\la}\cdots (b_i)_{\la}$.
Then probably the best-known example of a branching rule 
generalising \eqref{bP2} and resulting in symmetric polynomials is
\begin{equation}\label{Ok}
M_{\la}(x_1,\dots,x_n)=
\sum_{\mu\subseteq\la} \frac{(t^{n-1}/x_n)_{\la}}{(t^{n-1}/x_n)_{\mu}}\,
P_{\la/\mu}(x_n) M_{\mu}(x_1,\dots,x_{n-1}).
\end{equation}
The $M_{\la}(x)=M_{\la}(x;q,t)$ are the interpolation Macdonald polynomials
of Knop, Okounkov and Sahi \cite{Knop97,Okounkov97,Okounkov98,Sahi96}, 
and \eqref{Ok} is \cite[Proposition 5.3]{Okounkov98}.
For comparison with \cite{Okounkov98}, we have 
\[
M_{\la}(x_1,\dots,x_n)=t^{(n-1)\abs{\la}}
P_{\la}^{\ast}(t^{1-n}x_1,\dots,t^{-1}x_{n-1},x_n).
\]
{}From \eqref{Ok} it is clear that the top-homogeneous component 
of $M_{\la}(x)$ is the Macdonald polynomial $P_{\la}(x)$ so that
$\{M_{\la}|~l(\la)\leq n\}$ forms an inhomogeneous basis of
$\Lambda_{n,\F}$.

For $\la=(\la_1,\dots,\la_n)$ and $\mu=(\mu_1,\dots,\mu_{n-1})$
such that $\la-\mu$ is a horizontal strip and such that $\la_1\leq m$
denote by $m^n-\la$ and $m^{n-1}-\mu$ the partitions 
$(m-\la_n,\dots,m-\la_1)$ and $(m-\mu_{n-1},\dots,m-\mu_1)$.
Note that $(m^n-\la)-(m^{n-1}-\mu)$ is again a horizontal strip.
It follows from \eqref{psi} that
\[
P_{(m^n-\la)/(m^{n-1}-\mu)}(1/z;1/q,1/t) 
=z^{-m}P_{\la/\mu}(z;1/q,1/t) 
=z^{-m}P_{\la/\mu}(z)
\]
for $\la_1\leq m$.
It also follows from \eqref{psi} and \eqref{Ok} that
\[
\frac{M_{m^n}(1/x_1,\dots,1/x_n;1/q,1/t)}
{M_{m^{n-1}}(1/x_1,\dots,1/x_{n-1};1/q,1/t)}=
x_n^{-m} (x_n;1/q)_m.
\]
If we replace $(x,q,t)\mapsto (1/x,1/q,1/t)$ in \eqref{Ok} and
then change $\la\mapsto (m-\la_n,\dots,m-\la_1)$ and $\mu\mapsto
(m-\mu_{n-1},\dots,m-\mu_1)$ the branching rule for the interpolation
Macdonald polynomials may thus be recast as
\begin{multline}\label{MM}
\frac{M_{m^n-\la}(1/x_1,\dots,1/x_n;1/q,1/t)}
{M_{m^n}(1/x_1,\dots,1/x_n;1/q,1/t)} \\
=\sum_{\mu\subseteq\la} 
\frac{(q^{1-m}x_n/t)_{\mu}}{(q^{1-m}x_n)_{\la}}\, 
P_{\la/\mu}(x_n) \, 
\frac{M_{m^{n-1}-\mu}(1/x_1,\dots,1/x_{n-1};1/q,1/t)}
{M_{m^{n-1}}(1/x_1,\dots,1/x_{n-1};1/q,1/t)}
\end{multline}
In this paper we consider a generalisation 
$R_{\la}(x;b)=R_{\la}(x;b;q,t)$
of the Macdonald polynomials and the Macdonald interpolation
polynomials defined recursively by the branching rule
\begin{equation}\label{Rbf}
R_{\la}(x_1,\dots,x_n;b):=\sum_{\mu\subseteq\la} 
\frac{(bx_n/t)_{\mu}}{(bx_n)_{\la}}\, 
P_{\la/\mu}(x_n) R_{\mu}(x_1,\dots,x_{n-1};b).
\end{equation}

Our interest in the functions $R_{\la}(x;b)$ is not merely that they
provide another example of a class of symmetric functions defined by a simple
branching formula. Indeed, it may be shown that the more general
\begin{equation}\label{Rab}
R_{\la}(x_1,\dots,x_n;a,b):=\sum_{\mu\subseteq\la} 
\frac{(x_n/a)_{\la}(bx_n/t)_{\mu}}
{(x_n/a)_{\mu}(bx_n)_{\la}}\, P_{\la/\mu}(a) 
R_{\mu}(x_1,\dots,x_{n-1};at,b)
\end{equation}
is also symmetric. Moreover $R_{\la}(x;a,b)$ 
is a limiting case (reducing BC$_n$ symmetry to $\Symm_n$ symmetry
and breaking ellipticity) of Rains' BC$_n$ symmetric abelian interpolation
functions \cite{Rains03,Rains06} so that $R_{\la}(x;b)$ is not actually new. 
However, it turns out that the function $R_{\la}(x;b)$ has a number of
interesting properties, not shared by $R_{\la}(x;a,b)$
or the more general BC$_n$ abelian functions.
For example, $R_{\la}(x;b)$ satisfies a Pieri formula
that not just implies the standard Pieri formulas for
Macdonald polynomials, but also gives an $\sn$ generalisation
of the famous $q$-Gauss summation formula. Specifically, with $\R_{\la}(x;b)$ 
a suitable normalisation of the functions $R_{\la}(x;b)$ defined in
equation \eqref{Rnorm}, and $X$ an arbitrary finite alphabet, we prove that
\begin{equation}\label{qGaussI}
\sum_{\la} 
\Bigl(\frac{c}{ab}\Bigr)^{\abs{\la}} (a,b)_{\la} \R_{\la}(X;c)=
\prod_{x\in X} \frac{(cx/a,cx/b)_{\infty}}{(cx,cx/ab)_{\infty}}.
\end{equation}
For $X=\{1\}$ this simplifies to the well-known $q$-Gauss sum
\begin{equation}\label{gauss}
\sum_{k=0}^{\infty} 
\frac{(a,b)_k}{(q,c)_k}\, 
\Bigl(\frac{c}{ab}\Bigr)^k =
\frac{(c/a,c/b)_{\infty}}{(c,c/ab)_{\infty}}.
\end{equation}

\section{Preliminaries on Macdonald polynomials}

We begin with a remark about notation.
If $f$ is a symmetric function we will often write $f(X)$ with
$X=\{x_1,\dots,x_n\}$ (and refer to $X$ as an alphabet) instead of 
$f(x)$ with $x=(x_1,\dots,x_n)$, the latter notation being 
reserved for function that are not (a priori) symmetric. Following
this notation we also use $f(X+Y)$ where $X+Y$ denotes the (disjoint)
union of the alphabets $X$ and $Y$, and $f(X+z)$ where $X+z$ denotes
the alphabet $X$ with the single letter $z$ added.

\medskip 
In the following  we review some of the basics of Macdonald polynomial
theory, most of which can be found in \cite{Macdonald88,Macdonald95}.

Let $T_{q,x_i}$ be the $q$-shift operator acting on the variable $x_i$:
\[
(T_{q,x_i}f)(x_1,\dots,x_n)=f(x_1,\dots,x_{i-1},qx_i,x_{i+1},\dots,x_n).
\]
Then the Macdonald polynomials $P_{\la}(X)=P_{\la}(X;q,t)$
for $X=\{x_1,\dots,x_n\}$ are the unique polynomial eigenfunctions 
of the Macdonald operator
\begin{equation}\label{Dnc}
D_n(c):=\sum_{I\subseteq [n]}(-c)^{\abs{I}}t^{\binom{\abs{I}}{2}}
\prod_{\substack{i\in I \\ j\not\in I}}
\frac{tx_i-x_j}{x_i-x_j}
\prod_{i\in I} T_{q,x_i},
\end{equation}
where $[n]:=\{1,\dots,n\}$. Explicitly,
\begin{equation}\label{MacOp}
D_n(c) P_{\la}(X)=P_{\la}(X)\prod_{i=1}^n (1-cq^{\la_i}t^{n-i}).
\end{equation}
For later reference we state the coefficient of $c^1$ of this equation
separately; if
\begin{equation}\label{Dn1}
D_n^1:=\sum_{i=1}^n \biggl(\:
\prod_{j\neq i} \frac{tx_i-x_j}{x_i-x_j} \biggr) T_{q,x_i},
\end{equation}
then
\begin{equation}\label{MacOp2}
D_n^1 P_{\la}(X)=P_{\la}(X)\sum_{i=1}^n q^{\la_i}t^{n-i}.
\end{equation}

For each square $s=(i,j)\in\Z^2$ in the (Young) diagram of a partition 
(i.e., for each $i\in\{1,\dots,l(\la)\}$ and $j\in\{1,\dots,\la_i\}$,
the arm-length $a(s)$, arm-colength $a'(s)$, leg-length $l(s)$
and leg-colength $l'(s)$ are given by
\[
a(s)=\la_i-j,\qquad a'(s)=j-1
\]
and
\[
l(s)=\la'_j-i,\qquad l'(s)=i-1,
\]
where $\la'$ is the conjugate of $\la$, obtained by reflecting
the diagram of $\la$ in the main diagonal.
Note that the generalised $q$-shifted factorial \eqref{bla}
can be expressed in terms of the colengths as
\[
(b)_{\la}=\prod_{s\in\la}(1-bq^{a'(s)}t^{-l'(s)}).
\]
With the above notation we define the further $q$-shifted factorials
$c'_{\la}=c'_{\la}(q,t)$, $c_{\la}=c_{\la}(q,t)$ and
$b_{\la}=b_{\la}(q,t)$ as
\[
c'_{\la}:=\prod_{s\in\la} (1-q^{a(s)+1}t^{l(s)})
\and c_{\la}:=\prod_{s\in\la} (1-q^{a(s)}t^{l(s)+1})
\]
and
\[
b_{\la}:=\frac{c_{\la}}{c'_{\la}}.
\]
Then the Macdonald polynomials $Q_{\la}(X)=Q_{\la}(X;q,t)$ are defined as
\[
Q_{\la}(X):=b_{\la} P_{\la}(X).
\]
We also need the skew Macdonald polynomials $P_{\la/\mu}$ and
$Q_{\la/\mu}$ defined for $\mu\subseteq\la$ by
\begin{subequations}\label{PXYa}
\begin{align}
P_{\la}(X+Y)&=\sum_{\mu\subseteq\la} P_{\la/\mu}(Y) P_{\mu}(X) \\
Q_{\la}(X+Y)&=\sum_{\mu\subseteq\la} Q_{\la/\mu}(Y) Q_{\mu}(X).
\end{align}
\end{subequations}
Note that $P_{\la/0}=P_{\la}$ and $Q_{\la/0}=Q_{\la}$,
and that $P_{\la/\la}=Q_{\la/\la}=1$.
To simplify some later equations it will be useful to
extend the definitions of $P_{\la/\mu}$ and $Q_{\la/\mu}$
to all partition pairs $\la,\mu$ by setting $P_{\la/\mu}=Q_{\la/\mu}=0$
if $\mu\not\subseteq\la$.
{}From \eqref{Dnc}, \eqref{MacOp} and \eqref{PXYa} it follows that
for $a$ a scalar,
\begin{subequations}\label{hom}
\begin{align}
P_{\la/\mu}(aX)&=a^{\abs{\la-\mu}} P_{\la/\mu}(X) \\
Q_{\la/\mu}(aX)&=a^{\abs{\la-\mu}} Q_{\la/\mu}(X),
\end{align}
\end{subequations}
where $aX:=\{ax|~x\in X\}$.

For subsequent purposes it will be convenient to also introduce
normalised (skew) Macdonald polynomials $\P_{\la/\mu}$ and 
$\Q_{\la/\mu}$ as
\begin{subequations}\label{norm}
\begin{align}
\P_{\la/\mu}(X)&=t^{n(\la)-n(\mu)} 
\frac{c'_{\mu}}{c'_{\la}}\, P_{\la/\mu}(X)=
t^{n(\la)-n(\mu)} 
\frac{c_{\mu}}{c_{\la}}\, Q_{\la/\mu}(X) \\[2mm]
\Q_{\la/\mu}(X)&=t^{n(\mu)-n(\la)} 
\frac{c'_{\la}}{c'_{\mu}}\, Q_{\la/\mu}(X)=
t^{n(\mu)-n(\la)} \frac{c_{\la}}{c_{\mu}}\, P_{\la/\mu}(X),
\label{skewQP}
\end{align}
\end{subequations}
where 
\[
n(\la):=\sum_{s\in\la} l'(s)=\sum_{i=1}^{l(\la)}(i-1)\la_i.
\]
Note that no additional factors arise in the normalised form of
\eqref{PXYa}:
\begin{align*}
\P_{\la}(X+Y)&=\sum_{\mu} \P_{\la/\mu}(Y) \P_{\mu}(X) \\
\Q_{\la}(X+Y)&=\sum_{\mu} \Q_{\la/\mu}(Y) \Q_{\mu}(X),
\end{align*}
and that $\P_{\la/0}=\P_{\la}$ and $\Q_{\la/0}=\Q_{\la}$.
If we define the structure constants $\f^{\la}_{\mu\nu}=
\f^{\la}_{\mu\nu}(q,t)$ by 
\begin{equation}\label{struc}
\P_{\mu}(X)\P_{\nu}(X)=\sum_{\la} \f^{\la}_{\mu\nu}\P_{\la}(X),
\end{equation}
then
\[
\f^{\la}_{\mu\nu}=
t^{n(\mu)+n(\nu)-n(\la)} \frac{c'_{\la}}{c'_{\mu}c'_{\nu}}\, 
f^{\la}_{\mu\nu},
\]
with $f^{\la}_{\mu\nu}$ the $q,t$-Littlewood--Richardson coefficients.

Below we make use of some limited $\la$-ring notation (see
\cite{Lascoux03} for more details). Let $p_r$ be the $r$th 
power-sum symmetric function
\[
p_r(X):=\sum_{x\in X} x^r
\]
and let $p_{\la}:=\prod_{i\geq i} p_{\la_i}$. Then the
$p_{\la}(X)$ form a basis of $\Lambda_n$ where $n=\abs{X}$.
Given a symmetric function $f(X)$ we define
\[
f\biggl[\frac{a-b}{1-t}\biggr]:=\phi_{a,b}(f),
\]
where $\phi_{a,b}$ is the evaluation homomorphism given by
\begin{equation}\label{phi}
\phi_{a,b}(p_r)=\frac{a^r-b^r}{1-t^r}.
\end{equation}

In particular, $f[(1-t^n)/(1-t)]=f(1,t,\dots,t^{n-1})$ is known
as the principal specialisation, which we will also denote as
$f(\spec{0})$, and $f[1/(1-t)]=f(1,t,t^2,\dots)$.
{}From \cite[Equation (6.24)]{LRW08}
\[
\Q_{\la/\mu}\biggl[\frac{a-b}{1-t}\biggr]=a^{\abs{\la-\mu}}
\sum_{\nu}(b/a)_{\nu}\f_{\mu\nu}^{\la},
\]
which, by $\f_{0\nu}^{\la}=\delta_{\la \nu}$, also implies that
\begin{equation}\label{Qat}
\Q_{\la}\biggl[\frac{1-a}{1-t}\biggr]=(a)_{\la}\and
\P_{\la}\biggl[\frac{1-a}{1-t}\biggr]
=t^{2n(\la)} \frac{(a)_{\la}}{c_{\la}c'_{\la}}.
\end{equation}
This last equation of yields the well-known principal specialisation formula 
\[
\P_{\la}(\spec{0})=t^{2n(\la)} \frac{(t^n)_{\la}}{c_{\la}c'_{\la}}.
\]

The Cauchy identity for (skew) Macdonald polynomials is given by
\[
\sum_{\la}\P_{\la/\mu}(X)\Q_{\la/\nu}(Y)=
\Biggl(\: \prod_{x\in X} \prod_{y\in Y}
\frac{(txy)_{\infty}}{(xy)_{\infty}} \Biggr)
\sum_{\la}\P_{\nu/\la}(X)\Q_{\mu/\la}(Y).
\]
The product on the right-hand side may alternatively be expressed in
terms of the power-sum symmetric functions as
\[
\exp\biggl(\:\sum_{r=1}^{\infty} \frac{1}{r}\,\frac{1-t^r}{1-q^r}\, 
p_r(X)p_r(Y)\biggr).
\]
It thus follows from \eqref{phi} and \eqref{Qat}, as well as some
elementary manipulations, that application of $\phi_{a,c}$ (acting on $Y$) 
turns the Cauchy identity into
\begin{equation}\label{skewC}
\sum_{\la}\Q_{\la/\nu}\biggl[\frac{a-c}{1-t}\biggr] \P_{\la/\mu}(X)=
\biggl(\:\prod_{x\in X} \frac{(cx)_{\infty}}{(ax)_{\infty}} \biggr)
\sum_{\la}\Q_{\mu/\la}\biggl[\frac{a-c}{1-t}\biggr] \P_{\nu/\la}(X).
\end{equation}
For $\mu=\nu=0$ (followed by the substitution $X\to X/a$ and then 
$a\to c/a$) this is the $q$-binomial identity for Macdonald polynomials 
\cite{Kaneko96,Macdonald} 
\begin{equation}\label{qbt}
\sum_{\la} (a)_{\la} \P_{\la}(X)
=\prod_{x\in X} \frac{(ax)_{\infty}}{(x)_{\infty}}.
\end{equation}
For later reference we also state the more general $(\mu,\nu)=(0,\mu)$ 
instance of \eqref{skewC}
\begin{equation}\label{eqPieriII}
\P_{\mu}(X) \prod_{x\in X} \frac{(bx)_{\infty}}{(ax)_{\infty}}
=\sum_{\la} \Q_{\la/\mu}\biggl[\frac{a-b}{1-t}\biggr]\P_{\la}(X).
\end{equation}
For reasons outlined below we will refer to this as a Pieri formula.

Let $\phi_{\la/\mu}=\phi_{\la,\mu}(q,t)$ and 
$\psi'_{\la/\mu}=\psi'_{\la,\mu}(q,t)$ be defined by
\[
\phi_{\la/\mu}:=\frac{b_{\la}}{b_{\mu}}\,\psi_{\la/\mu}
\quad\text{and}\quad
\psi'_{\la,\mu}(q,t):=\psi_{\la'/\mu'}(t,q).
\]
(For combinatorial expressions for all of
$\psi_{\la/\mu},\psi'_{\la/\mu}$ and $\phi_{\la/\mu}$, see \cite{Macdonald95}).
Further let $g_{(r)}(X):=P_{(r)}(X)\,(t)_r/(q)_r$ and $e_r(X)$ the $r$th
elementary symmetric function.
Then the Macdonald polynomials $P_{\la}(X)$ satisfy the 
Pieri formulas
\begin{subequations}\label{Pierir}
\begin{align}
P_{\mu}(X)g_r(X)
&=\sum_{\substack{\la\succcurlyeq\mu \\ \abs{\la-\mu}=r}}
\phi_{\la/\mu}P_{\la}(X) \\
P_{\mu}(X)e_r(X)
&=\sum_{\substack{\la'\succcurlyeq\mu' \\ \abs{\la-\mu}=r}}
\psi'_{\la/\mu} P_{\la}(X).
\end{align}
\end{subequations}
Now observe that \eqref{eqPieriII} for $b=at$ yields
\[
\P_{\mu}(X) \prod_{x\in X} \frac{(atx)_{\infty}}{(ax)_{\infty}}
=\sum_{\la} a^{\abs{\la-\mu}} 
\Q_{\la/\mu}(1)\P_{\la}(X),
\]
whereas for $a=bq$ it yields
\[
\P_{\mu}(X) \prod_{x\in X} (1-bx)
=\sum_{\la} b^{\abs{\la-\mu}}
\Q_{\la/\mu}\biggl[\frac{q-1}{1-t}\biggr]\P_{\la}(X).
\]
Since 
\[
\sum_{r\geq 0} a^r g_r(X)=
\prod_{x\in X}\frac{(atx)_{\infty}}{(ax)_{\infty}} 
\quad\text{and}\quad
\sum_{r\geq 0} (-b)^r e_r(X)=\prod_{x\in X} (1-bx)
\]
we therefore have
\begin{align*}
P_{\mu}(X) g_r(X)&=
\sum_{\abs{\la-\mu}=r} 
Q_{\la/\mu}(1)P_{\la}(X) \\
\intertext{and}
P_{\mu}(X) e_r(X)&=
(-1)^r \sum_{\abs{\la-\mu}=r} 
Q_{\la/\mu}\biggl[\frac{q-1}{1-t}\biggr]P_{\la}(X),
\end{align*}
where we have also used \eqref{skewQP}.
Identifying 
\[
Q_{\la/\mu}(1)=\begin{cases}
\phi_{\la/\mu} &\text{if $\mu\preccurlyeq\la$} \\[1mm]
0 &\text{otherwise}
\end{cases}
\]
and
\begin{equation}\label{Qpsip}
Q_{\la/\mu}\biggl[\frac{q-1}{1-t}\biggr]=
\begin{cases}\displaystyle
(-1)^{\abs{\la-\mu}} \psi'_{\la/\mu} &\text{if $\mu'\preccurlyeq\la'$} \\[1mm]
0 &\text{otherwise}
\end{cases}
\end{equation}
these two formulas are equivalent to the Pieri rules of \eqref{Pierir}.

\medskip

The skew polynomials can be used to define generalised $q$-binomial 
coefficients \cite{Okounkov97,Lassalle98,Lassalle99} as
\begin{equation}\label{qbinlamu}
\qbin{\la}{\mu}=
\qbin{\la}{\mu}_{q,t}:=
\Q_{\la/\mu}\biggl[\frac{1}{1-t}\biggr].
\end{equation}
In particular has $\qbin{\la}{\mu}=0$ if $\mu\not\subseteq\la$
and
\[
\qbin{(m)}{(k)}=\prod_{i=1}^k \frac{1-q^{i+m-k}}{1-q^i}=\qbin{m}{k}
\]
with on the right the classical $q$-binomial coefficients $\qbin{m}{k}=
\qbin{m}{k}_q$.

If $\la_{(i)}:=(\la_1,\dots,\la_{i-1},\la_i-1,\la_{i+1},\dots,\la_i)$
then \cite{Lassalle98}
\[
(1-q)t^{i-1} \qbin{\la}{\la_{(i)}}=\frac{c'_{\la}}{c'_{\la_{(i)}}}\,
\psi'_{\la/\la_{(i)}}.
\]
This, together with \cite[Th\'eor\`eme 9, Bis]{Lassalle98}
\begin{equation}\label{rec}
\bigl(\omega_{\la}-\omega_{\mu}\bigr)
\qbin{\la}{\mu}
=(1-q)\sum_{i=1}^n q^{-\la_i}t^{i-n}
\qbin{\la}{\la_{(i)}}
\qbin{\la_{(i)}}{\mu},
\end{equation}
where 
\begin{equation}\label{omega}
\omega_{\la}=\omega_{\la}(q,t):=\sum_{i=1}^n q^{-\la_i} t^{i-n},
\end{equation}
provides a simple recursive method to compute the
generalised $q$-binomial coefficients.

\begin{lemma}\label{lemLas}
Assume that $l(\la),l(\mu)\leq n$. Then
\begin{subequations}
\begin{align}\label{tisq}
\qbin{\la}{\mu}&=
\frac{s_{\mu}(\spec{0})}{s_{\la}(\spec{0})}
\det_{1\leq i,j\leq n}\biggl(
\qbin{\la_i+n-i}{\mu_j+n-j}\biggr) &&
\text{if $t=q$} \\
\intertext{and}
\label{tis1}
\qbin{\la}{\mu}&=\sum_{u^{+}=\mu}
\prod_{i=1}^n \qbin{\la_i}{u_i} &&
\text{if $t=1$.}
\end{align}
\end{subequations}
\end{lemma}
In the above $\sum_{u^{+}=\mu}$ denotes a sum over compositions
$u\in\Nat^n$ in the $\Symm_n$ orbit of $\mu$. 

\begin{proof}
Assume that $t=q$. Then \eqref{qbinlamu} simplifies to
\begin{equation}\label{qbintq}
\qbin{\la}{\mu}=q^{n(\mu)-n(\la)}\,
\frac{c_{\la}}{c_{\mu}}\, s_{\la/\mu}\biggl[\frac{1}{1-q}\biggr],
\end{equation}
where we have also used \eqref{skewQP} and the fact that for $t=q$ the
(skew) Macdonald polynomials reduce to the (skew) Schur functions.
Let $h_r(X)$ be the $r$th complete symmetric function.
By application of the Jacobi--Trudi identity 
\cite[Equation (I.5.4)]{Macdonald95} 
\[
s_{\la/\mu}=\det_{1\leq i,j\leq n}(h_{\la_i-\mu_j-i+j})
\qquad \text{for $n\geq l(\la)$},
\]
and the principal specialisation formula
\cite[page 44]{Macdonald95}
\begin{equation}\label{pps}
s_{\la}(\spec{0})=\frac{q^{n(\la)}}{c'_{\la}}
\prod_{i=1}^n \frac{(q)_{\la_i+n-i}}{(q)_{n-i}}
\end{equation}
the generalised $q$-binomial coefficient \eqref{qbintq} can be expressed
as a determinant
\[
\qbin{\la}{\mu}=
\frac{s_{\mu}(\spec{0})}{s_{\la}(\spec{0})}\,
\det_{1\leq i,j\leq n}
\biggl(\frac{(q;q)_{\la_i+n-i}}{(q,q)_{\mu_j+n-j}}\,
h_{\la_i-\mu_j-i+j}\biggl[\frac{1}{1-q}\biggr]\biggr).
\]
Since
\[
h_r\biggl[\frac{1}{1-q}\biggr]=\frac{1}{(q)_r}
\]
this establishes the first claim.

The second claim follows in analogous manner.
Since making the substitution $t=1$ in the right-hand side
of \eqref{qbinlamu} is somewhat problematic it is best to first
use the symmetry  \cite[Equation (2.12)]{Okounkov97}
\begin{equation}\label{qtinv}
\qbin{\la}{\mu}_{q,t}=\qbin{\la'}{\mu'}_{t^{-1},q^{-1}}.
\end{equation}
Since
\[
P_{\la/\mu}(X;1,t)=e_{\la'/\mu'}(X)
\]
we therefore get
\[
\qbin{\la}{\mu}_{q^{-1},1}=
\qbin{\la'}{\mu'}_{1,q} 
=q^{n(\mu')-n(\la')}\,
\frac{c_{\la'}(1,q)}{c_{\mu'}(1,q)}\,
e_{\la/\mu}\biggl[\frac{1}{1-q}\biggr].
\]
Using 
\[
e_{\la/\mu}=\sum_{u^{+}=\mu} \prod_{i=1}^n e_{\la_i-u_i}
\qquad \text{for $n\geq l(\la)$},
\]
and 
\[
e_r\biggl[\frac{1}{1-q}\biggr]=\frac{q^{\binom{r}{2}}}{(q;q)_r}
\]
as well as $c_{\la'}(q,1)=\prod_i (q)_{\la_i}$, it follows that
\[
\qbin{\la}{\mu}_{q^{-1},1}=
\sum_{u^{+}=\mu} \prod_{i=1}^n q^{-\la_i u_i}\qbin{\la_i}{u_i}.
\]
Finally replacing $q\mapsto 1/q$ yields the second claim.
\end{proof}

\section{Symmetric functions and branching rules}

In this section we consider the question posed in the introduction:
\begin{quote}
\textit{Can one find new(?) branching-type formulas,
similar to \eqref{bSchur}, \eqref{bP} and \eqref{Ok}, 
that lead to symmetric functions?}
\end{quote}

Assume that $k$ is a fixed nonnegative integer, and let
$\mathbf{a}=(a_1,a_2,\dots,a_k)$
denote a finite sequence of parameters.
Then we are looking for branching coefficients
$f_{\la/\mu}(z;\mathbf{a})$ such that
\begin{subequations}
\begin{equation}\label{frule}
f_{\la}(x_1,\dots,x_n;\mathbf{a})=\sum_{\mu\subseteq\la} 
f_{\la/\mu}(x_n;\mathbf{a})
f_{\mu}(x_1,\dots,x_{n-1};\mathbf{a}'),
\end{equation}
subject to the initial condition
\begin{equation}\label{init}
f_{\la}(\text{--}\,;\mathbf{a})=\delta_{\la,0}
\end{equation}
\end{subequations}
defines a symmetric function.
In the above
$\mathbf{a}'=(a_1',\dots,a_k')=g(\mathbf{a})$.
Of course, \eqref{frule} for $n=1$ combined with \eqref{init} 
implies that
\[
f_{\la}(z;\mathbf{a})=f_{\la/0}(z;\mathbf{a}).
\]
If one wishes to only consider symmetric functions with the
standard property
\begin{equation}\label{boundary}
f_{\la}(x_1,\dots,x_n;\mathbf{a})=0 \qquad \text{if $~l(\la)>n$}
\end{equation}
then the additional condition 
\[
f_{\la/\mu}(z;\mathbf{a})=0 \qquad \text{if $~l(\la)-l(\mu)>1$}
\]
must be imposed.

Because we assume the branching coefficients to be independent of $n$,
it may perhaps seem we are excluding interesting classes of symmetric functions
such as the interpolation Macdonald polynomials.
As will be shown shortly, assuming $n$-independence is not actually
a restriction, and \eqref{Ok} may easily be recovered as a special
case of \eqref{frule}.

Now let us assume that \eqref{frule} yields a symmetric function
$f_{\la}(x_1,\dots,x_n;\mathbf{a})$ for all $n\leq N$.
(For $N=0$ and $N=1$ this is obviously not an assumption.)
Then, 
\begin{align*}
f_{\la}(x_1,\dots,x_{n-1},y,z;\mathbf{a})&=
\sum_{\mu\subseteq\la} f_{\la/\mu}(z;\mathbf{a})
f_{\mu}(x_1,\dots,x_{n-1},y;\mathbf{a}') \\
&=\sum_{\nu\subseteq\mu\subseteq\la} 
f_{\la/\mu}(z;\mathbf{a})
f_{\mu/\nu}(y;\mathbf{a}')
f_{\nu}(x_1,\dots,x_{n-1};\mathbf{a}'')
\end{align*}
is a symmetric function in $x_1,\dots,x_{n-1},y$ (for $n\leq N$).
For it to also be a symmetric function in $x_1,\dots,x_{n-1},y,z$
we must have 
\[
f_{\la}(x_1,\dots,x_{n-1},y,z;\mathbf{a})=
f_{\la}(x_1,\dots,x_{n-1},z,y;\mathbf{a}),
\]
implying that for fixed $\la$ 
\begin{multline*}
\sum_{\nu\subseteq\mu\subseteq\la} 
f_{\la/\mu}(z;\mathbf{a})
f_{\mu/\nu}(y;\mathbf{a}')
f_{\nu}(x_1,\dots,x_{n-1};\mathbf{a}'') \\
=\sum_{\nu\subseteq\mu\subseteq\la} 
f_{\la/\mu}(y;\mathbf{a})
f_{\mu/\nu}(z;\mathbf{a}')
f_{\nu}(x_1,\dots,x_{n-1};\mathbf{a}''),
\end{multline*}
where $\mathbf{a}'':=g(\mathbf{a}')$.
Hence a sufficient condition for \eqref{frule} to yield
a symmetric function is
\begin{equation}\label{suf}
\sum_{\nu\subseteq\mu\subseteq\la} 
f_{\la/\mu}(z;\mathbf{a})
f_{\mu/\nu}(y;\mathbf{a}')=
\sum_{\nu\subseteq\mu\subseteq\la} 
f_{\la/\mu}(y;\mathbf{a})
f_{\mu/\nu}(z;\mathbf{a}')
\end{equation}
for partitions $\la,\nu$ such that $\nu\subseteq\la$.

As a first example let us show how to recover the Macdonald
interpolation polynomials of the introduction. 
To this end we take $\mathbf{a}=(a)$, 
$\mathbf{a'}=(a/t)$, 
and
\[
f_{\la/\mu}(z;\mathbf{a})=
f_{\la/\mu}(z;a)=
\frac{(a/z)_{\la}}{(a/z)_{\mu}}\, P_{\la/\mu}(z).
\]
Clearly, the resulting polynomials $f_{\la}(x;a)$ correspond to the
interpolation polynomials after the specialisation $a=t^{n-1}$.
To see that \eqref{suf} is indeed satisfied we substitute the
above choice for the branching coefficient (recall the convention
that $P_{\la/\mu}:=0$ if $\mu\not\subseteq\la$) to obtain
\begin{equation}\label{Ok2}
\sum_{\mu} 
\frac{(a/z)_{\la}(a/ty)_{\mu}}{(a/z)_{\mu}(a/ty)_{\nu}}\,
P_{\la/\mu}(z) P_{\mu/\nu}(y) 
=\sum_{\mu}
\frac{(a/y)_{\la}(a/tz)_{\mu}}{(a/y)_{\mu}(a/tz)_{\nu}}\,
P_{\la/\mu}(y) P_{\mu/\nu}(z).
\end{equation}
The identity \eqref{Ok2} is easily proved using Rains'
Sears transformation for skew Macdonald polynomials
\cite[Corollary 4.9]{Rains05}
\begin{multline}\label{RainsS2}
\sum_{\mu}
\frac{(aq/b,aq/c)_{\la}(d,e)_{\mu}}{(aq/b,aq/c)_{\mu}(d,e)_{\nu}} \,
P_{\la/\mu}\biggl[\frac{1-aq/de}{1-t}\biggr]
P_{\mu/\nu}\biggl[\frac{aq/de-a^2q^2/bcde}{1-t}\biggr] \\
=\sum_{\mu}
\frac{(aq/d,aq/e)_{\la}(b,c)_{\mu}}{(aq/d,aq/e)_{\mu}(b,c)_{\nu}} \,
P_{\la/\mu}\biggl[\frac{1-aq/bc}{1-t}\biggr]
P_{\mu/\nu}\biggl[\frac{aq/bc-a^2q^2/bcde}{1-t}\biggr].
\end{multline}
After simultaneously replacing $(a,b,c,d,e)\mapsto 
(c,a/tz,cqz/a,a/ty,cqy/a)$ and taking the $c\to\infty$ limit
we obtain \eqref{Ok2}.

If, more generally, we let $(a,b,c,d,e)\mapsto (c,a/bz,cqz/a,a/by,cqy/a)$ 
in \eqref{RainsS2} and take the $c\to\infty$ limit we find that
\begin{multline*}
\sum_{\mu}
\frac{(a/z)_{\la}(a/by)_{\mu}}{(a/z)_{\mu}(a/by)_{\nu}}\,
P_{\la/\mu}\biggl[\frac{z-bz}{1-t}\biggr]
P_{\mu/\nu}\biggl[\frac{y-by}{1-t}\biggr] \\
=\sum_{\mu}
\frac{(a/y)_{\la}(a/bz)_{\mu}}{(a/y)_{\mu}(a/bz)_{\nu}}\,
P_{\la/\mu}\biggl[\frac{y-by}{1-t}\biggr]
P_{\mu/\nu}\biggl[\frac{z-bz}{1-t}\biggr].
\end{multline*}
The Macdonald interpolation polynomials may thus be generalised by
taking $\mathbf{a}=(a,b)$, $\mathbf{a'}=(a/b,b)$ and
\[
f_{\la/\mu}(z;\mathbf{a})=f_{\la/\mu}(z;a,b)=
z^{\abs{\la-\mu}} \frac{(a/z)_{\la}}{(a/z)_{\mu}}\,
P_{\la/\mu}\biggl[\frac{1-b}{1-t}\biggr].
\]

\begin{proposition}
The polynomials
$M_{\la}(x_1,\dots,x_n;a,b)=M_{\la}(x_1,\dots,x_n;a,b;q,t)$
defined by 
\[
M_{\la}(x_1,\dots,x_n;a,b)=
\sum_{\mu} 
x_n^{\abs{\la-\mu}} \frac{(a/x_n)_{\la}}{(a/x_n)_{\mu}}\,
P_{\la/\mu}\biggl[\frac{1-b}{1-t}\biggr]
M_{\mu}(x_1,\dots,x_{n-1};a/b,b)
\]
subject to 
$M_{\la}(\text{--}\,;a,b)=\delta_{\la,0}$
are symmetric. Moreover, the interpolation Macdonald polynomials
corresponds to
\[
M_{\la}(x_1,\dots,x_n)=
M_{\la}(x_1,\dots,x_n;t^{n-1},t).
\]
\end{proposition}
The polynomials $M_{\la}(x_1,\dots,x_n;a,b)$ are an example of a class
of symmetric functions for which $l(\la)>n$ does not imply vanishing.
For example,
\[
M_{\la}(z;a,b)=
z^{\abs{\la}} (a/z)_{\la}\, P_{\la}\biggl[\frac{1-b}{1-t}\biggr]=
t^{n(\la)} z^{\abs{\la}} \frac{(a/z,b)_{\la}}{c_{\la}}.
\]

The next example corresponds to Okounkov's BC$_n$ symmetric
interpolation polynomials \cite{Okounkov98b} 
(see also \cite{Okounkov03,Rains05}).
\begin{proposition}\label{PropBCn}
If we take take $\mathbf{a}=(a,b)$, $\mathbf{a}'=(a/t,b/t)$ and 
\[
f_{\la/\mu}(z;\mathbf{a})=
f_{\la/\mu}(z;a,b)
=\frac{(a/z,bz)_{\la}}{(a/z,bz)_{\mu}}\, P_{\la/\mu}(1/b)
\]
in \eqref{frule} then the resulting functions
$f_{\la}(x;a,b)=f_{\la}(x;a,b;q,t)$ are symmetric.
\end{proposition}
Writing $O_{\la}(x;a,b)$ instead of $f_{\la}(x;a,b)$,
the (Laurent) polynomials $O_{\la}(x;a,b)$ satisfy the
symmetries
\[
O_{\la}(x;a,b)=\Bigl(\frac{a}{b}\Bigr)^{\abs{\la}}O_{\la}(1/x;b,a)=
\Bigl(\frac{a}{b}\Bigr)^{\abs{\la}}q^{2n(\la')}
O_{\la}(1/x;1/a,1/b;1/q,1/t).
\]
(This follows easily using that $P_{\la}(X;1/q,1/t)=P_{\la}(X,q,t)$.)
Moreover, Okounkov's BC$_n$ interpolation Macdonald
polynomials $P_{\la}^{\ast}(x;q,t,s)$ follow as
\[
P_{\la}^{\ast}(x_1,tx_2,\dots,t^{n-1}x_n;q,t,s)
=q^{-n(\la')} O_{\la}(x;1,s^2 t^{2(n-1)};q,t).
\]
(Since $O_{\la}(ax;a,b/a)=a^{\abs{\la}}O_{\la}(x;1,b)$ the 
$O_{\la}(x;a,b)$ are not more general than the
$P_{\la}^{\ast}(x;q,t,s)$.)

\begin{proof}[Proof of Proposition \eqref{PropBCn}]
Substituting the claim in \eqref{suf} and using \eqref{hom} gives
\begin{multline*}
\sum_{\mu} \frac{(a/z,bz)_{\la} (a/yt,byt)_{\mu}}
{(a/z,bz)_{\mu}(a/yt,byt)_{\nu}}\, 
P_{\la/\mu}(1) P_{\mu/\nu}(t) \\
=\sum_{\mu} 
\frac{(a/y,by)_{\la}(a/zt,bzt)_{\mu}}
{(a/y,by)_{\mu}(a/zt,bzt)_{\nu}}\, 
P_{\la/\mu}(1) P_{\mu/\nu}(t).
\end{multline*}
This is \eqref{RainsS2} with
$(a,b,c,d,e)\mapsto (ab/qt,bz/t,a/zt,a/yt,by/t)$.
\end{proof}

Our final example will (in the limit) lead to the functions studied
in the remainder of the paper.
\begin{proposition}\label{propRab}
If we take $\mathbf{a}=(a,b)$, $\mathbf{a}'=(at,b)$ and
\[
f_{\la/\mu}(z;\mathbf{a})=
f_{\la/\mu}(z;a,b)
=\frac{(z/a)_{\la}(bz/t)_{\mu}}
{(z/a)_{\mu}(bz)_{\la}}\, P_{\la/\mu}(a) 
\]
in \eqref{frule} then the resulting functions
$f_{\la}(x;a,b)=f_{\la}(x;a,b;q,t)$ are symmetric.
\end{proposition}

\begin{proof}
Substituting the claim in \eqref{suf} and using \eqref{hom} gives
\begin{multline*}
\sum_{\mu}
\frac{(z/a,by)_{\la}(bz/t,y/at)_{\mu}}
{(z/a,by)_{\mu}(bz/t,y/at)_{\nu}}\, 
P_{\la/\mu}(1) P_{\mu/\nu}(t)\\
=\sum_{\mu}
\frac{(y/a,bz)_{\la}(by/t,z/at)_{\mu}}
{(y/a,bz)_{\mu}(by/t,z/at)_{\nu}}\, 
P_{\la/\mu}(1) P_{\mu/\nu}(t).
\end{multline*}
This is \eqref{RainsS2} with 
$(a,b,c,d,e)\mapsto (byz/aqt,by/t,z/at,bz/t,y/at)$.
\end{proof}

If we write $R_{\la}(x;a,b)$ instead of $f_{\la}(x;a,b)$
the symmetric functions of Proposition~\ref{propRab}
correspond to the functions described by the branching rule 
\eqref{Rab} of the introduction. 
As already mentioned there, the $R_{\la}(x;a,b)$ are not new, 
and follow as a special limiting case of
much more general functions studied by Rains \cite{Rains03,Rains06}. 
More specifically, Rains defined a family of abelian interpolation 
functions
\[
R^{\ast(n)}_{\la}(x;a,b)=R^{\ast(n)}_{\la}(x;a,b;q,t;p),
\]
where $x=(x_1,\dots,x_n)$. The $R^{\ast(n)}_{\la}(x)$ are BC$_n$ symmetric
and, apart from parameters $a,b,q,t$, depend on an elliptic nome $p$.
In \cite[Theorem 4.16]{Rains06} Rains proved the branching rule
\[
R^{\ast(n+1)}_{\la}(x_1,\dots,x_{n+1};a,b)=\sum_{\mu\subseteq\la} 
c_{\la\mu}^{(n)}(x_{n+1};a,b) R^{\ast(n)}_{\mu}(x_1,\dots,x_n;a,b),
\]
where the branching coefficient
$c_{\la\mu}^{(n)}(z;a,b)=c^{(n)}_{\la\mu}(z;a,b;q,t;p)$
is expressed in terms of the elliptic binomial coefficient 
$\qbinE{\la}{\mu}_{[a,b](v_1,\dots,v_k)}$ 
(see \cite[Equation (4.2)]{Rains03}) as
\begin{equation}\label{clamu}
c_{\la\mu}^{(n)}(z;a,b)=\qbinE{\la}{\mu}_{[at^n/b,t](azt^n,at^n/z,pqa/tb)}.
\end{equation}
If we define
\begin{align*}
R_{\la}(x;a,b)&=
R_{\la}(x;a,b;q,t) \\ &=
\Bigl(\frac{t^{1-n}}{b}\Bigr)^{\abs{\la}}
P_{\la}(\spec{0})
\lim_{p\to 0}
R^{\ast(n)}_{\la}(p^{1/4}x;p^{1/4}at^{n-1},p^{3/4}b/q;1/q,1/t;p)
\end{align*}
and compute the corresponding limit of \eqref{clamu}
we obtain the branching rule \eqref{Rab} with $n\mapsto n+1$.

\section{The symmetric function $R_{\la}(x;b)$}
In the remainder of the paper we consider the symmetric function
\begin{align*}
R_{\la}(X;b)
&=(-1)^{\abs{\la}} q^{-n(\la')} t^{n(\la)} \lim_{a\to 0} R_{\la}(X;a,b) \\
&=q^{-n(\la')} t^{n(\la)} 
\Bigl(-\frac{t^{1-n}}{b}\Bigr)^{\abs{\la}}
P_{\la}(\spec{0})  \hspace{2cm} (n:=\abs{X}) \\
&\qquad\qquad \times\lim_{a\to 0} \lim_{p\to 0}
R^{\ast(n)}_{\la}(p^{1/4}X;p^{1/4}at^{n-1},p^{3/4}b/q;1/q,1/t;p)
\end{align*}
which, alternatively, is defined by the branching rule \eqref{Rbf}.
Because $R_{\la}$ is a limiting case of the abelian interpolation
function $R^{\ast(n)}_{\la}$ many properties of former follow
by taking appropriate limits in the results of \cite{Rains03,Rains06}.
For example, it follows from \cite[Proposition 3.9]{Rains06} 
that the
$R_{\la}(X;b)$ for $X=\{x_1,\dots,x_n\}$
satisfy a $q$-difference equation generalising \eqref{MacOp}.
Specifically, with $D_n(b,c)$ the generalised Macdonald operator
\[
D_n(b,c)=\sum_{I\subseteq [n]}(-1)^{\abs{I}}t^{\binom{\abs{I}}{2}}
\prod_{\substack{i\in I \\ j\not\in I}} \frac{tx_i-x_j}{x_i-x_j}
\prod_{j\not\in I}(1-bx_j)
\prod_{i\in I} (c-bt^{1-n}x_i)T_{q,x_i}
\]
we have
\begin{equation}\label{Dnbc}
D_n(b,c) R_{\la}(X;b)=R_{\la}(X;bq)\prod_{i=1}^n (1-cq^{\la_i}t^{n-i}).
\end{equation}

\medskip

Below we will first prove a number of elementary properties of the
functions $R_{\la}(X;b)$ using only the branching rule \eqref{Rbf}.
Like the previous result, most of these can also be 
obtained by taking appropriate limits in results of Rains 
for the abelian interpolation functions
$R^{\ast(n)}_{\la}(X;a,b)$.
Then we give several deeper results for $R_{\la}(X;b)$
(such as Theorem~\ref{thmskewRC} and Corollaries~\ref{Pieri},
\ref{CorGauss} and \ref{CorKTW}) 
that, to the best of our knowledge, have no analogues for 
$R^{\ast(n)}_{\la}(X;a,b)$ or $R_{\la}(X;a,b)$.
First however we restate the branching rule \eqref{Rbf}
in the equivalent form
\begin{equation}\label{R3}
R_{\la}(X;b)=\sum_{\mu} \frac{(bz/t)_{\mu}}{(bz)_{\la}}\, 
P_{\la/\mu}(z) R_{\mu}(Y;b),
\end{equation}
where $X=Y+z$.

When $X=\{z\}$ we find from \eqref{R3} that
\begin{equation}\label{n1}
R_{(k)}(z;b)=\frac{z^k}{(bz)_k}.
\end{equation}
{}From this it is clear that $R_{(k)}(cz;b)=c^k R_{(k)}(z;bc)$ 
and that in the $c\to\infty$ limit $R_{(k)}(cz;b)$ is given by 
$(-b)^{-k} q^{-\binom{k}{2}}$.
It also shows that
\[
R_{(k+1)}(z;b)=R_{(k)}(z;bq).
\]
All three statements easily generalise to arbitrary $X$.

\begin{lemma}\label{Lembindep}
For $c$ a scalar,
\[
R_{\la}(cX;b)=c^{\abs{\la}}R_{\la}(X;bc).
\]
\end{lemma}

\begin{lemma}\label{Lemclim}
For $c$ a scalar and $n:=\abs{X}$,
\[
\lim_{c\to\infty} R_{\la}(cX;b)=
\Bigl(-\frac{t^{1-n}}{b}\Bigr)^{\abs{\la}}
q^{-n(\la')}t^{n(\la)}
P_{\la}(\spec{0}).
\]
\end{lemma}

\begin{lemma}\label{Lemreduce}
Let $n:=\abs{X}$ and $\la$ a partition such that 
$l(\la)=n$. Define $\mu:=(\la_1-1,\dots,\la_n-1)$. Then
\[
R_{\la}(X;b)=R_{\mu}(X;bq)\prod_{x\in X} \frac{x}{1-bx}.
\]
\end{lemma}
This last result allows the definition of
$R_{\la}(X;b)$ to be extended to all weakly decreasing
integer sequences $\la=(\la_1,\dots,\la_n)$.

\begin{proof}[Proof of Lemmas~\ref{Lembindep}--\ref{Lemreduce}]
By \eqref{n1} all three statements are obviously true for $X$ a
single-letter alphabet, and we proceed by induction on $n$,
the cardinality of $X$.

By \eqref{R3},
\begin{equation}\label{Rcyb}
R_{\la}(cX;b)=\sum_{\mu} \frac{(bcz/t)_{\mu}}{(bcz)_{\la}}\, 
P_{\la/\mu}(cz) R_{\mu}(cY;b).
\end{equation}
Using \eqref{hom} and the appropriate induction hypothesis this yields
\[
R_{\la}(cX;b)=c^{\abs{\la}}
\sum_{\mu} \frac{(bcz/t)_{\mu}}{(bcz)_{\la}}\, 
P_{\la/\mu}(z) \, R_{\mu}(Y;bc) \\
=c^{\abs{\la}} R_{\la}(X;bc),\notag
\]
establishing the first lemma.

Taking the $c\to\infty$ limit on both sides of \eqref{Rcyb} and then
using induction we get
\begin{align*}
\lim_{c\to\infty} R_{\la}(cX;b)
&=\sum_{\mu} (-b)^{\abs{\mu}-\abs{\la}} q^{n(\mu')-n(\la')} 
t^{n(\la)-n(\mu)-\abs{\mu}}
P_{\la/\mu}(1) \, \lim_{c\to\infty} R_{\mu}(cY;b) \\
&=\Bigl(-\frac{t^{1-n}}{b}\Bigr)^{\abs{\la}} q^{-n'(\la)}t^{n(\la)}
\sum_{\mu}
P_{\la/\mu}(t^{n-1}) P_{\mu}(t^{n-2},\dots,t,1) \\
&=\Bigl(-\frac{t^{1-n}}{b}\Bigr)^{\abs{\la}} q^{-n'(\la)}t^{n(\la)}
P_{\la}(\spec{0}),
\end{align*}
where the last equality follows from \eqref{PXYa}.

To prove the final lemma we consider \eqref{Rcyb} with $c=1$
and, in accordance with the conditions of Lemma~\ref{Lemreduce},
with $\la_n\geq 1$. 
Since $P_{\la/\nu}(a)$ vanishes unless $\la-\nu$ is a horizontal
strip this implies that $\nu_{n-1}\geq 1$.
The summand also vanishes if $l(\nu)>n-1$ so that we
we may assume that $l(\nu)=n-1$.
Defining $\eta=(\nu_1-1,\dots,\nu_{n-1}-1)$ and
$\mu=(\la_1-1,\dots,\la_n-1)$ and
using induction, as well as
\[
\frac{(bz/t)_{\nu}}{(bz)_{\la}}=
\frac{1}{1-bz}\, \frac{(bzq/t)_{\eta}}{(bzq)_{\mu}}
\qquad\text{and}\qquad
P_{\la/\nu}(z)=z P_{\mu/\eta}(z),
\]
we get
\begin{align*}
R_{\la}(X;b)
&=\frac{z}{1-bz} \biggl(\:\prod_{x\in Y} \frac{x}{1-bx}\biggr)
\sum_{\eta\subseteq\mu}
\frac{(bzq/t)_{\eta}}{(bzq)_{\mu}}\, P_{\mu/\eta}(z) R_{\eta}(Y;bq) \\
&=R_{\mu}(X;bq)\,\prod_{x\in X}\frac{x}{1-bx}.
\end{align*}
where in the final step we have used \eqref{R3} and $X=Y+z$.
\end{proof}

\begin{proposition}[Principal specialisation]\label{PSP}
For $\la$ such that $l(\la)\leq n$,
\[
R_{\la}(\spec{0};b)=\frac{P_{\la}(\spec{0})}{(bt^{n-1})_{\la}}\,
=\frac{t^{n(\la)}(t^n)_{\la}}{(bt^{n-1})_{\la}\, c_{\la}}.
\]
\end{proposition}
By Lemma~\ref{Lembindep} this may be stated slightly more
generally as
\begin{equation}\label{eqpsp2}
R_{\la}(a\spec{0};b)=\frac{P_{\la}(a\spec{0})}
{(abt^{n-1})_{\la}}.
\end{equation}

\begin{proof}
Iterating \eqref{Rbf} using
\[
\sum_{\nu} P_{\la/\nu}(X)P_{\nu/\mu}(Y)=P_{\la/\mu}(X+Y),
\]
we obtain the generalised branching rule
\[
R_{\la}(x_1,\dots,x_m,t^{n-1},\dots,t,1;b)
=\sum_{\mu} \frac{(b/t)_{\mu}}{(bt^{n-1})_{\la}}\, 
P_{\la/\mu}(\spec{0}) R_{\mu}(x_1,\dots,x_m;b),
\]
for $l(\la)\leq n+m$.
When $m=0$ this results in the claim.
\end{proof}

\begin{proposition}[Evaluation symmetry] 
For $\la$ such that $l(\la)\leq n$ set 
\[
\spec{\la}=(q^{\la_1}t^{n-1},\dots,q^{\la_{n-1}}t,1).
\]
Then
\[
\frac{R_{\la}(a\spec{\mu};b)}{R_{\la}(a\spec{0};b)}=
\frac{R_{\mu}(a\spec{\la};b)}{R_{\mu}(a\spec{0};b)}.
\]
\end{proposition}

\begin{proof} 
We may view the evaluation symmetry as a rational function identity in
$b$. Hence it suffices to give a proof for $b=q^{1-m}$ where $m$ runs
over all integers such that $\la_1,\mu_1\leq m$.
But (see \eqref{MM} and \eqref{Rbf})
\[
R_{\la}(X;q^{1-m})=
\frac{M_{m^n-\la}(1/x_1,\dots,1/x_n;1/q,1/t)}
{M_{m^n}(1/x_1,\dots,1/x_n;1/q,1/t)}
\]
and
\[
M_{m^n}(1/X;1/q,1/t)=\prod_{x\in X} x^m (x;1/q)_m
\]
so that we need to prove that
\[
\frac{M_{m^n-\la}(a\spec{\mu})}{M_{m^n-\la}(a\spec{0})}=
\frac{M_{m^n-\mu}(a\spec{\la})}{M_{m^n-\mu}(a\spec{0})}\,
\frac{(aqt^{n-1})_{\mu}}{(aqt^{n-1})_{\la}}\,
\frac{(aq^{1-m}t^{n-1})_{\la}}{(aq^{1-m}t^{n-1})_{\mu}}.
\]
Making the substitutions $\la\mapsto m^n-\la$, $\mu\mapsto m^n-\mu$
and $a\mapsto aq^{-m}t^{1-n}$ we get
\[
\frac{M_{\la}(a/\spec{\mu})}{M_{\la}(aq^{-m}t^{1-n}\spec{0})}=
\frac{M_{\mu}(a/\spec{\la})}{M_{\mu}(aq^{-m}t^{1-n}\spec{0})}\,
\frac{(q^mt^{n-1}/a)_{\mu}}{(q^mt^{n-1}/a)_{\la}}\,
\frac{(t^{n-1}/a)_{\la}}{(t^{n-1}/a)_{\mu}}\,q^{m(\abs{\la}-\abs{\mu}}.
\]
Finally, by the principal specialisation formula for the
interpolation Macdonald polynomials \cite{Okounkov97},
\[
\frac{M_{\la}(aq^{-m}t^{1-n}\spec{0})}
{M_{\la}(a/\spec{0})}=
\frac{(q^mt^{n-1}/a)_{\la}}{(t^{n-1}/a)_{\la}}\,q^{-m\abs{\la}}
\]
so that we end up with
\[
\frac{M_{\la}(a/\spec{\mu})}{M_{\la}(a/\spec{0})}=
\frac{M_{\mu}(a/\spec{\la})}{M_{\mu}(a/\spec{0})}.
\]
This is the known evaluation symmetry of the interpolation Macdonald
polynomials \cite[Section 2]{Okounkov97}.
\end{proof}

It is clear from \eqref{R3} that $R_{\la}(X;0)=P_{\la}(X)$
with on the right a Macdonald polynomial.
The Macdonald polynomials in turn generalise the Jack polynomials
$P^{(\alpha)}_{\la}(X)$, since
$P^{(\alpha)}_{\la}(X)=\lim_{q\to 1} P_{\la}(X;q,q^{1/\alpha})$.
Combining the last two equations it thus follows that
\[
P^{(\alpha)}_{\la}(X)=\lim_{q\to 1} R_{\la}(X;0;q,q^{1/\alpha}).
\]
Curiously, there is an alternative path from $R_{\la}(X;b)$ to the Jack
polynomials as follows. For $X$ an alphabet let
\[
\hat{X}:=\Bigl\{ \Bigl(\frac{x}{1-x}\Bigr) \,\Big|~x\in X\Bigr\}.
\]
\begin{proposition}\label{propjack}
We have
\[
P^{(\alpha)}_{\la}(\hat{X})=\lim_{q\to 1} R_{\la}(X;1;q,q^{1/\alpha}).
\]
\end{proposition}

\begin{proof}
Let $X=Y+z$ be a finite alphabet.

Replacing $(b,t)\mapsto (1,q^{1/\alpha})$ in \eqref{R3}
and taking the $q\to 1$ limit yields
\[
R^{(\alpha)}_{\la}(X)=\sum_{\mu} (1-z)^{\abs{\mu}-\abs{\la}}
P_{\la/\mu}^{(\alpha)}(z) R^{(\alpha)}_{\mu}(Y),
\]
where $P_{\la/\mu}^{(\alpha)}$ is a skew Jack polynomial and
\[
R^{(\alpha)}_{\la}(X):=\lim_{q\to 1} R_{\la}(X;1;q,q^{1/\alpha}).
\]
Using the homogeneity of $P_{\la/\mu}^{(\alpha)}$ the above can be 
rewritten as
\[
R^{(\alpha)}_{\la}(X)=\sum_{\mu} 
P_{\la/\mu}^{(\alpha)}\Bigl(\frac{z}{1-z}\Bigr) R^{(\alpha)}_{\mu}(Y).
\]
Comparing this with 
\[
P^{(\alpha)}_{\la}(X)=\sum_{\mu} 
P_{\la/\mu}^{(\alpha)}(z) P^{(\alpha)}_{\mu}(Y)
\]
the proposition follows.
\end{proof}

\section{Cauchy, Pieri and Gauss formulas for $R_{\la}(X;b)$}

Probably our most important new results for $R_{\la}(X;b)$
are generalisation of the skew Cauchy identity \eqref{skewC},
the Pieri formula \eqref{eqPieriII} 
and the $q$-Gauss formula \eqref{qGaussI}. 

Before we get to these result we first need a few more definitions.
First of all, in analogy with \eqref{norm}, we set
\begin{equation}\label{Rnorm}
\R_{\la}(X;b):=t^{n(\la)} \frac{R_{\la}(X;b)}{c'_{\la}}
\end{equation}
so that \eqref{R3} becomes
\[
\R_{\la}(X;b)=\sum_{\mu} \frac{(bz/t)_{\mu}}{(bz)_{\la}}\, 
\P_{\la/\mu}(z) \R_{\mu}(Y;b).
\]
Furthermore, we also define the skew functions $\R_{\la/\mu}(X;b)$ by
\begin{equation}\label{SkewR}
\R_{\la/\mu}(X;b):=\sum_{\nu} \frac{(bz/t)_{\nu}}{(bz)_{\la}}\, 
\P_{\la/\nu}(z) \R_{\mu/\nu}(Y;b)
\end{equation}
and 
\[
\R_{\la/\mu}(\text{--}\,;b)=\delta_{\la \mu}.
\]
In other words,
\[
\R_{\la/\mu}(X+Y;b)=\sum_{\nu} \R_{\la/\mu}(X;b) \R_{\mu/\nu}(Y;b)
\]
and $\R_{\la/\mu}(X;0)=\P_{\la/\mu}(X)$.

\begin{theorem}[Skew Cauchy-type identity]\label{thmskewRC}
Let $ab=cd$ and $X$ a finite alphabet. Then
\begin{multline}\label{skewRC}
\sum_{\la} \frac{(b/c)_{\la}}{(b/c)_{\nu}}\,
\Q_{\la/\nu}\biggl[\frac{a-c}{1-t}\biggr] \R_{\la/\mu}(X;b) \\
=\biggl(\: \prod_{x\in X} \frac{(cx,dx)_{\infty}}{(ax,bx)_{\infty}}
\biggr)
\sum_{\la} \frac{(b/c)_{\mu}}{(b/c)_{\la}}\,
\Q_{\mu/\la}\biggl[\frac{a-c}{1-t}\biggr] \R_{\nu/\la}(X;d).
\end{multline}
\end{theorem}
Note that for $b=0$ the theorem simplifies to \eqref{skewC}.
We defer the proof of \eqref{skewRC} till the end of this section and
first list a number of corollaries.

\begin{corollary}[Pieri formula]\label{Pieri}
Let $ab=cd$ and $X$ a finite alphabet. Then
\begin{equation}\label{PieriR}
\R_{\mu}(X;d)\prod_{x\in X} 
\frac{(cx,dx)_{\infty}}{(ax,bx)_{\infty}} 
=\sum_{\la} \frac{(b/c)_{\la}}{(b/c)_{\mu}}\,
\Q_{\la/\mu}\biggl[\frac{a-c}{1-t}\biggr] \R_{\la}(X;b).
\end{equation}
\end{corollary}
This follows from the theorem by taking $\mu=0$ and then replacing
$\nu$ by $\mu$.

When $b,d\to 0$ equation \eqref{PieriR} yields the Pieri
formula \eqref{eqPieriII} for Macdonald polynomials.
When $b,c\to 0$ and $a\to 1$ such that $b/c=d$ equation~\eqref{PieriR}
yields (after replacing $d\mapsto b$)
\begin{equation}\label{snqbt}
\sum_{\la} \frac{(b)_{\la}}{(b)_{\mu}}\,
\qbin{\la}{\mu} \P_{\la}(X)=
\R_{\mu}(X;b)\prod_{x\in X} 
\frac{(bx)_{\infty}}{(x)_{\infty}}.
\end{equation}
For $\mu=0$ this is the $q$-binomial formula for Macdonald polynomials
\eqref{qbt}, and for $b=0$ it is Lassalle's \cite{Lassalle98}
\[
\sum_{\la} \qbin{\la}{\mu} \P_{\la}(X)=
\P_{\mu}(X)\prod_{x\in X} \frac{1}{(x)_{\infty}}.
\]

The Jack polynomial limit of \eqref{snqbt} is of particular interest.
To concisely state this we need some more notation.
Let
\[
\binom{\la}{\mu}^{\!\!(\alpha)}:=\lim_{q\to 1} 
\qbin{\la}{\mu}_{q,q^{1/\alpha}}
\]
or, alternatively \cite{Kaneko93,Lassalle90,OO97},
\[
\frac{P_{\la}^{(\alpha)}(x_1+1,\dots,x_n+1)}
{P_{\la}^{(\alpha)}(1^n)}=
\sum_{\mu} \binom{\la}{\mu}^{\!\!(\alpha)}
\frac{P_{\mu}^{(\alpha)}(x_1,\dots,x_n)}
{P_{\mu}^{(\alpha)}(1^n)},
\]
where $n$ is any integer such that $n\geq l(\la)$.
Further let
\[
(b;\alpha)_{\la}:=\prod_{i\geq 1} (b+(1-i)/\alpha))_{\la_i}
\]
with $(b)_k=b(b+1)\cdots(b+k-1)$,
\[
c'_{\la}(\alpha):=\prod_{s\in\la}(a(s)+1+l(s)/\alpha)
\]
and
\[
\P_{\la}^{(\alpha)}(X):=\frac{P_{\la}^{(\alpha)}(X)}{c'_{\la}(\alpha)}.
\]
Using all of the above, replacing $(b,q,t)$ in \eqref{snqbt} by
$(q^{\beta},q,q^{1/\alpha})$ and taking the (formal) limit $q\to 1$
with the aid of Proposition~\ref{propjack}, we arrive at the
following identity.

\begin{corollary}[Binomial formula for Jack polynomials]
For $X$ a finite alphabet 
\[
\sum_{\la} \frac{(\beta;\alpha)_{\la}}{(\beta;\alpha)_{\mu}} 
\binom{\la}{\mu}^{\!\!(\alpha)}
\P_{\la}^{(\alpha)}(X)=
\P_{\mu}^{(\alpha)}(\hat{X})
\prod_{x\in X} \frac{1}{(1-x)^{\beta}}.
\]
\end{corollary}

Another special case of \eqref{PieriR} worth stating
is the following multivariable extension of the $_1\phi_1$ summation
\cite[II.5]{GR04}, which follows straightforwardly by taking the
$a,d\to 0$ limit,
\begin{equation}\label{1p1}
\sum_{\la}
c^{\abs{\la-\mu}}\frac{(b/c)_{\la}}{(b/c)_{\mu}}\,
\Q_{\la/\mu}\biggl[\frac{0-1}{1-t}\biggr] \R_{\la}(X;b)=
\P_{\mu}(X)\prod_{x\in X} 
\frac{(cx)_{\infty}}{(bx)_{\infty}}.
\end{equation}
This provides an expansion of the right-hand side different from 
\eqref{eqPieriII}. 

If we let $\nu=0$ in Theorem~\ref{thmskewRC}, use \eqref{Qat} and then
replace $(a,b,c)\mapsto(c/ab,c,c/a)$ we obtain
\[
\sum_{\la} 
\Bigl(\frac{c}{ab}\Bigr)^{\abs{\la-\mu}} 
\frac{(a,b)_{\la}}{(a,b)_{\mu}} \, \R_{\la/\mu}(X;c)=
\prod_{x\in X} \frac{(cx/a,cx/b)_{\infty}}{(cx,cx/ab)_{\infty}}.
\]
For $\mu=0$ we state this separately.

\begin{corollary}[$\sn$ $q$-Gauss sum]\label{CorGauss}
For $X$ a finite alphabet
\begin{equation}\label{newGauss}
\sum_{\la} 
\Bigl(\frac{c}{ab}\Bigr)^{\abs{\la}} 
(a,b)_{\la} \R_{\la}(X;c)=
\prod_{x\in X} \frac{(cx/a,cx/b)_{\infty}}{(cx,cx/ab)_{\infty}}.
\end{equation}
\end{corollary}
As mentioned in the introduction, for $X=\{1\}$ this simplifies to the 
standard $q$-Gauss sum \eqref{gauss} thanks to \eqref{n1}. 
More generally, if we principally specialise 
$X=t^{1-n}\spec{0}=\{1,t^{-1},\dots,t^{1-n}\}$ and use \eqref{eqpsp2} the
$\sn$ $q$-Gauss sum simplifies to Kaneko's $q$-Gauss sum
for Macdonald polynomials \cite[Proposition 5.4]{Kaneko96}
\[
\sum_{\la} \Bigl(\frac{ct^{1-n}}{ab}\Bigr)^{\abs{\la}} 
\frac{(a,b)_{\la}}{(c)_{\la}}\, P_{\la}(\spec{0})
=\prod_{i=1}^n \frac{(ct^{1-i}/a,ct^{1-i}/b)_{\infty}}
{(ct^{1-i},ct^{1-i}/ab)_{\infty}}.
\]

As another consequence of the theorem we obtain an explicit
expression for the Taylor series of $\R_{\mu}(X;b)$ in $b$.
For $\mu\subseteq\la$ let $(1)_{\la/\mu}$ be defined as
$(1)_{\la/\mu}=\lim_{a\to 1} (a)_{\la}/(a)_{\mu}$. That is
\[
(1)_{\la/\mu}=\prod_{s\in\la-\mu}(1-q^{a'(s)}t^{-l'(s)}).
\]

\begin{corollary}\label{corbTaylor}
We have
\begin{equation}\label{bTaylor}
\R_{\mu}(X;b)=\sum_{\la\supseteq\mu}b^{\abs{\la-\mu}}
(1)_{\la/\mu}\qbin{\la}{\mu} \P_{\la}(X),
\end{equation}
or, equivalently,
\[
[b^r] \R_{\mu}(X;b)=\sum_{\substack{\la\supseteq\mu \\ \abs{\la-\mu}=r}}
(1)_{\la/\mu}\qbin{\la}{\mu} \P_{\la}(X).
\]
\end{corollary}

\begin{proof}
Replacing $(c,X)\mapsto(a/b,bX)$ in Lemma~\ref{Lembindep} and
expressing the resulting identity in terms of the normalised
function $\R_{\la}(X;b)$ (see \eqref{Rnorm}) we get
\begin{equation}\label{absymm}
\R_{\la}(aX;b)=\Bigl(\frac{a}{b}\Bigr)^{\abs{\la}}\R_{\la}(bX;a).
\end{equation}
Combined with \eqref{snqbt} this implies that
\[
\R_{\mu}(aX;b)=\Bigl(\frac{a}{b}\Bigr)^{\abs{\mu}}
\biggl(\:\prod_{x\in X} \frac{(bx)_{\infty}}{(abx)_{\infty}}
\biggr)
\sum_{\la} b^{\abs{\la}}
\frac{(a)_{\la}}{(a)_{\mu}} \qbin{\la}{\mu} \P_{\la}(X).
\]
The summand vanishes unless $\mu\subseteq\la$ and so we may add
this as a restriction in the sum over $\la$. Then the limit
$a\to 1$ limit may be taken without causing ambiguities, and
the claim follows.
\end{proof}

Corollary~\ref{corbTaylor} implies the following simple expressions 
for $R_{\la}(X;b)$ when $t=q$ (Schur-like case) or $t=1$ (monomial-like case).
\begin{proposition}
Let $X=\{x_1,\dots,x_n\}$.
Then
\begin{subequations}
\begin{align}\label{PPtisq}
R_{\la}(X;b)&=\frac{1}{\Delta(X)} 
\det_{1\leq i,j\leq n}\biggl(\;
\frac{x_i^{\la_j+n-j}}{(bx_i)_{\la_j-j+1}} \biggr) && \text{if $t=q$} \\
\intertext{and}
\label{PPtis1}
R_{\la}(X;b)&=
\sum_{u^{+}=\la} \biggl(\:\prod_{i=1}^n 
\frac{x_i^{u_i}}{(bx_i)_{u_i}} \biggr) && \text{if $t=1$.}
\end{align}
\end{subequations}

\end{proposition}

\begin{proof}
Since the two claims are proved in almost identical fashion 
we only present a proof of \eqref{PPtisq}.
The only significant difference is that the omitted proof of 
\eqref{PPtis1} uses \eqref{tis1} instead of \eqref{tisq}.

Assume that $t=q$. Let $\nu=\la-(0,1,\ldots,n-1)$ and suppose that
$\nu_n\geq 0$. Since for any $k\geq 0$, one has the expansion  
\[
\frac{x^{k+n-1}}{(bx)_k}=
\sum_{r=0}^{\infty} x^{k+n-1+r} \qbin{k+r-1}{r},
\]
the matrix $(x_i^{\nu_j+n-1}/(bx_i)_{\nu_j})$
factorises into the product of rectangular matrices
\[
\big( x_i^{n-1+r} \big)_{\substack{i=1,\ldots,n \\ r=0,1,\ldots}}
\and
\bigg(b^{r-k} \qbin{k-1}{r-k} \bigg)_{\substack{r=0,1,\ldots \\ 
k=\nu_1,\ldots,\nu_n}}.
\]
According to Cauchy--Binet theorem, the determinant on the right-hand 
side of \eqref{PPtisq} factorises into a sum of products of minors of 
these two matrices.

On the other hand, by \eqref{alt} and \eqref{tisq}, the
expansion \eqref{bTaylor} gives
\[
R_{\la}(X;b)\Delta(X)   =
\sum_{\mu\supseteq\la} b^{\abs{\mu-\la}}(1)_{\mu/\la}
\det_{1\leq i,j\leq n}\biggl(
\frac{1}{(q)_{\mu_i-\la_j-i+j}}\biggr)
\det_{1\leq i,j\leq n}\bigl(x_i^{\mu_j+n-j}\bigr),
\]
where we have also used \eqref{pps} and \eqref{Rnorm}.
Using $\nu$ instead of $\la$, and $\eta=\mu-(0,1,\ldots,n-1)$,
this becomes 
\[
R_{\la}(X;b)\Delta(X)   =
\sum_{\eta\supset\nu} \det_{1\leq i,j\leq n}\biggl(
\qbin{\eta_i-1}{\nu_j-1} b^{\eta_i-\nu_j}\biggr)
\det_{1\leq i,j\leq n}\bigl(x_i^{\eta_j}\bigr) \ , 
\]
which is precisely the Cauchy--Binet expansion.
The restriction $\nu_n\geq 0$ is lifted using Lemma~\ref{Lemreduce}.
\end{proof}

Recall that the Macdonald polynomials are the eigenfunctions of the
operator $D_n^1$, see \eqref{MacOp2}. Because 
$P_{\la}(X;q,t)=P_{\la}(X;q^{-1},t^{-1})$ this can also be stated as
\[
\D_n^1 P_{\la}(X)=\omega_{\la} P_{\la}(X),
\]
where $\D_n^1:=D_n^1(q^{-1};t^{-1})$ and $\omega_{\la}$ is given 
in \eqref{omega}.

A second consequence of Corollary~\ref{corbTaylor} it a generalisation
of this identity as follows.
Let 
\[
A_i(x;t):=
\prod_{\substack{j=1 \\ j\neq i}}^n \frac{tx_i-x_j}{x_i-x_j}
\]
and
\[
\D_n^1(b):=\sum_{i=1}^n A_i(x;t^{-1})
\biggl(\Bigl(1-\frac{bx_i}{q}\Bigr)
T_{q^{-1},x_i}+\frac{bx_i}{q}\biggr),
\]
so that $\D_n^1(0)=\D_n^1$.

\begin{theorem}
We have
\[
\D_n^1(b) R_{\la}(X;b)=\omega_{\la} R_{\la}(X;b).
\]
\end{theorem}

\begin{proof}
Define the operator $\E_n$ as
\[
\E_n:=
\sum_{i=1}^n x_i A_i(x;t^{-1}) \bigl(T_{q^{-1},x_i}-1\bigr).
\]
Combining the $(q,t)\mapsto (q^{-1},t^{-1})$ instance of 
\cite[Proposition 9]{Lassalle98} with \eqref{qtinv} gives
\[
\E_n\P_{\la}(X)=(1-q)\sum_{i=1}^n 
q^{-\la_i}t^{i-n}(1)_{\la^{(i)}/{\la}}
\qbin{\la^{(i)}}{\la} \P_{\la^{(i)}}(X),
\]
where
$\la^{(i)}:=(\la_1,\dots,\la_{i-1},\la_i+1,\la_{i+1},\dots,\la_i)$.

Since
\[
\D_n^1(b)=\D_n^1(0)-bq^{-1}\E_n
\]
this implies that
\[
\D_n^1(b)\P_{\la}(X)=\omega_{\la} \P_{\la}(X)
-b(1-q)\sum_{i=1}^n q^{-\la_i-1}t^{i-n}(1)_{\la^{(i)}/\la}
\qbin{\la^{(i)}}{\la} \P_{\la^{(i)}}(X).
\]
By Corollary~\ref{corbTaylor} we can now compute the action of 
$\D_n^1(b)$ on $\R_{\la}(X;b)$:
\begin{align*}
\D_n^1(b) \R_{\mu}(X;b)
&=\sum_{\la\supseteq\mu}b^{\abs{\la-\mu}}
(1)_{\la/\mu}\qbin{\la}{\mu} \D_n^1(b) \P_{\la}(X) \\
&=\sum_{\la\supseteq\mu}b^{\abs{\la-\mu}}
(1)_{\la/\mu}\qbin{\la}{\mu} 
\omega_{\la} \P_{\la}(X) \\
&\quad-(1-q) \sum_{i=1}^n \sum_{\la\supseteq\mu}
b^{\abs{\la-\mu}+1} q^{-\la_i-1}t^{i-n}(1)_{\la^{(i)}/\mu}
\qbin{\la^{(i)}}{\la} \qbin{\la}{\mu} \P_{\la^{(i)}}(X) \\
&=\sum_{\la\supseteq\mu}b^{\abs{\la-\mu}}
(1)_{\la/\mu}\qbin{\la}{\mu} 
\omega_{\la} \P_{\la}(X) \\
&\quad-(1-q) \sum_{i=1}^n \sum_{\la\supseteq\mu}
b^{\abs{\la-\mu}} q^{-\la_i}t^{i-n}(1)_{\la/\mu}
\qbin{\la}{\la_{(i)}} \qbin{\la_{(i)}}{\mu} \P_{\la}(X),
\end{align*}
where we have also used that $(1)_{\la/\mu}(1)_{\mu/\nu}=(1)_{\la/\nu}$ for
$\nu\subseteq\mu\subseteq\la$. 
Recalling the recursion \eqref{rec}, the sum over $i$ on the right
can be performed to give
\[
\D_n^1(b) \R_{\mu}(X;b)
=\omega_{\mu}
\sum_{\la\supseteq\mu} 
b^{\abs{\la-\mu}} (1)_{\la/\mu} \qbin{\la}{\mu} \P_{\la}(X).
\]
Again using \eqref{bTaylor} completes the proof.
\end{proof}

As a third and final application of 
Corollary~\ref{corbTaylor} we derive a simple
expression for $D_n(b,c)P_{\mu}(x)$, to be compared with
\eqref{MacOp} (obtained for $b=0$) or with \eqref{Dnbc}.

\begin{proposition}
We have
\[
D_n(b,c)P_{\mu}(x) =\sum_{\la'\succcurlyeq\mu'} (-b)^{\abs{\la-\mu}}
P_{\la}(x) \psi'_{\la/\mu} \prod_{\la_i=\mu_i} (1-cq^{\la_i}t^{n-i})
\prod_{\la_i\neq\mu_i} (1-q^{\mu_i}t^{1-i}).
\]
\end{proposition}

\begin{proof}
First we use \eqref{skewQP} and \eqref{Qpsip}, as well as the fact that
for $\la-\mu$ a vertical strip (i.e., $\la_i-\mu_i=0,1$)
\[
\prod_{\la_i=\mu_i} (1-cq^{\la_i}t^{n-i})
\prod_{\la_i\neq\mu_i} (1-q^{\mu_i}t^{1-i})
=\frac{(cqt^{n-1})_{\mu}}{(cqt^{n-1})_{\la}}\, (1)_{\la/\mu}
\prod_{i=1}^n  (1-cq^{\la_i}t^{n-i}),
\]
to put the proposition in the form
\[
D_n(b,c)\P_{\mu}(x)
=\biggl(\:\prod_{i=1}^n  (1-cq^{\la_i}t^{n-i})\biggr)
\sum_{\la\supseteq\mu} b^{\abs{\la-\mu}}
\frac{(cqt^{n-1})_{\mu}}{(cqt^{n-1})_{\la}}\,
(1)_{\la/\mu}
\Q_{\la/\mu}\biggl[\frac{q-1}{1-t}\biggr] \P_{\la}(x).
\]
Since $M_{\mu\nu}:=\qbin{\mu}{\nu}$ is lower-triangular
(with $M_{\mu\mu}=1$), it is invertible (for the explicit inverse see
\cite[page 540]{Okounkov97}).
Since $(1)_{\la/\mu}(1)_{\mu/\nu}=(1)_{\la/\nu}$ for
$\nu\subseteq\mu\subseteq\la$ the above equation is thus equivalent to  
\begin{multline}\label{DPP}
D_n(b,c)\sum_{\mu\supseteq\nu}b^{\abs{\mu-\nu}}
(1)_{\mu/\nu}
\qbin{\mu}{\nu} \P_{\mu}(x) \\
=\biggl(\:\prod_{i=1}^n  (1-cq^{\la_i}t^{n-i})\biggr)
\sum_{\la\supseteq\mu\supseteq\nu} b^{\abs{\la-\nu}}
\frac{(cqt^{n-1})_{\mu}}{(cqt^{n-1})_{\la}}\,(1)_{\la/\nu}
\qbin{\mu}{\nu} \Q_{\la/\mu}\biggl[\frac{q-1}{1-t}\biggr]
\P_{\la}(x)
\end{multline}
for fixed $\nu$.

Taking $(a,b,c,d)\mapsto (cq,c,0,1)$ in \cite[Corollary 4.9]{Rains05}
in Rains' generalised $q$-Pfaff--Saalsch\"utz sum \cite[Corollary 4.9]{Rains05}
\begin{equation}\label{Saal}
\sum_{\mu} \frac{(a)_{\mu}}{(c)_{\mu}}\,
\Q_{\la/\mu}\biggl[\frac{a-b}{1-t}\biggr]
\Q_{\mu/\nu}\biggl[\frac{b-c}{1-t}\biggr]=
\frac{(a)_{\nu}(b)_{\la}}{(b)_{\nu}(c)_{\la}}\,
\Q_{\la/\nu}\biggl[\frac{a-c}{1-t}\biggr]
\end{equation}
yields
\[
\sum_{\mu}
\frac{(cq)_{\mu}}{(cq)_{\nu}}\,
\qbin{\mu}{\nu} 
\Q_{\la/\mu}\biggl[\frac{q-1}{1-t}\biggr]
=q^{\abs{\la-\nu}} \frac{(c)_{\la}}{(c)_{\nu}}
\qbin{\la}{\nu}.
\]
This allows the sum over $\mu$ on the right of \eqref{DPP} to be
carried out, leading to
\begin{align*}
D_n(b,c)&\sum_{\mu\supseteq\nu}b^{\abs{\mu-\nu}}
(1)_{\mu/\nu} \qbin{\mu}{\nu} \P_{\mu}(x) \\
&=\biggl(\:\prod_{i=1}^n  (1-cq^{\la_i}t^{n-i})\biggr)
\sum_{\la\supseteq\nu} (bq)^{\abs{\la-\nu}}
\frac{(ct^{n-1})_{\la}(cqt^{n-1})_{\nu}}
{(ct^{n-1})_{\nu}(cqt^{n-1})_{\la}}\,(1)_{\la/\nu}
\qbin{\la}{\nu} \P_{\la}(x) \\
&=\biggl(\:\prod_{i=1}^n  (1-cq^{\nu_i}t^{n-i})\biggr)
\sum_{\la\supseteq\nu} (bq)^{\abs{\la-\nu}}
(1)_{\la/\nu} \qbin{\la}{\nu} \P_{\la}(x).
\end{align*}
By Corollary~\ref{corbTaylor} this is the same as 
\eqref{Dnbc}.
\end{proof}

\subsection{Proof of Theorem~\ref{thmskewRC}}
To prove the theorem we first prepare the following result.
\begin{proposition}\label{Propskew}
For $\mu,\nu$ partitions
\begin{multline*}
\sum_{\la} \frac{(a)_{\la}}{(c)_{\la}}\,
\P_{\la/\mu}\Bigl(\frac{c}{ab}\Bigr)
\Q_{\la/\nu}\biggl[\frac{1-b}{1-t}\biggr] \\
=\frac{(a)_{\mu}}{(c/t)_{\mu}}\, \frac{(a)_{\nu}}{(c/b)_{\nu}}\, 
\frac{(c/a,c/b)_{\infty}}{(c,c/ab)_{\infty}}
\sum_{\la} \frac{(c/bt)_{\la}}{(a)_{\la}}\,
\P_{\nu/\la}\Bigl(\frac{c}{ab}\Bigr)
\Q_{\mu/\la}\biggl[\frac{1-b}{1-t}\biggr].
\end{multline*}
\end{proposition}
Note that for $\nu=0$ this is
\[
\sum_{\la} \frac{(a,b)_{\la}}{(c)_{\la}}\,
\P_{\la/\mu}\Bigl(\frac{c}{ab}\Bigr)
=\frac{(c/a,c/b)_{\infty}}{(c,c/ab)_{\infty}}\,
\frac{(a,b)_{\mu}}{(c/t)_{\mu}},
\]
which, for $\mu=0$, simplifies to the $q$-Gauss sum \eqref{gauss}.

\begin{proof}[Proof of Proposition~\ref{Propskew}]
Key is the Cauchy-type identity for skew Macdonald polynomials
due to Rains \cite[Corollary 3.8]{Rains08}:
\begin{multline}\label{Rains}
\frac{1}{Z}
\sum_{\la} q^{\abs{\la}} \frac{(a,b)_{\la}}{(e,f)_{\la}}\,
\P_{\la/\mu}\biggl[\frac{1-c}{1-t}\biggr]
\Q_{\la/\nu}\biggl[\frac{1-d}{1-t}\biggr] \\
=\Bigl(\frac{q}{c}\Bigr)^{\abs{\mu}}\frac{(a,b)_{\mu}}{(e/c,f/c)_{\mu}}\,
\Bigl(\frac{q}{d}\Bigr)^{\abs{\nu}} \frac{(a,b)_{\nu}}{(e/d,f/d)_{\nu}} 
\qquad\qquad\qquad\qquad\qquad\\
\times \sum_{\la} \Bigl(\frac{cd}{q}\Bigr)^{\abs{\la}}
\frac{(e/cd,f/cd)_{\la}}{(a,b)_{\la}}\,
\P_{\nu/\la}\biggl[\frac{1-c}{1-t}\biggr]
\Q_{\mu/\la}\biggl[\frac{1-d}{1-t}\biggr],
\end{multline}
provided that the sum on the left terminates and 
the balancing condition $abcdq=eft$ holds.
The prefactor $Z$ refers to the sum on the left for $\mu=\nu=0$, i.e., to
\begin{align*}
Z&=\sum_{\la} q^{\abs{\la}} \frac{(a,b)_{\la}}{(e,f)_{\la}}\,
\P_{\la}\biggl[\frac{1-c}{1-t}\biggr]
\Q_{\la}\biggl[\frac{1-d}{1-t}\biggr] \\
&=\sum_{\la} q^{\abs{\la}} \frac{(a,b,c,d)_{\la}}{(e,f)_{\la}}\,
\P_{\la}\biggl[\frac{1}{1-t}\biggr].
\end{align*}

Since $(t)_{\la}$ vanishes if $l(\la)>1$, and since
\[
(t)_{(k)} \P_{(k)}\biggl[\frac{1}{1-t}\biggr]=\frac{1}{(q)_k}
\]
if follows that for $(b,c)=(q^{-N},t)$
\[
Z=\sum_{k=0}^N \frac{(a,d,q^{-N})_k\, q^k}{(q,e,adq^{1-N}/e)_k}
=\frac{(e/a,e/d)_N}{(e,e/ad)_N},
\]
where the second equality follows from the $q$-Pfaff--Saalsch\"utz sum 
\cite[Equation (II.12)]{GR04}.

To make the same $(b,c)=(q^{-N},t)$ specialisation in the right-hand side
of \eqref{Rains} we first replace the sum over $\la$ by $\la\subseteq
\nu$ (using the fact that $P_{\nu/\la}(X)=0$ if
$\la\not\subseteq\nu$). Then
\[
\frac{(b)_{\nu}}{(b)_{\la}}=\prod_{s\in\nu-\la}(1-bq^{a'(s)}t^{-l'(s)})
\]
is well-defined for $b=q^{-N}$. It thus follows that for for
$(b,c)=(q^{-N},t)$ \eqref{Rains} simplifies to
\begin{multline*}
\sum_{\la} q^{\abs{\la}} \frac{(a,q^{-N})_{\la}}{(e,f)_{\la}}\,
\P_{\la/\mu}(1)
\Q_{\la/\nu}\biggl[\frac{1-d}{1-t}\biggr] \\
=\Bigl(\frac{q}{t}\Bigr)^{\abs{\mu}}\frac{(a,q^{-N})_{\mu}}{(e/t,f/t)_{\mu}}\,
\Bigl(\frac{q}{d}\Bigr)^{\abs{\nu}}\frac{(a,q^{-N})_{\nu}}{(e/d,f/d)_{\nu}}\, 
\frac{(e/a,e/d)_N}{(e,e/ad)_N}
\qquad\qquad\qquad\qquad\qquad\\
\times \sum_{\la\subseteq\nu} \Bigl(\frac{dt}{q}\Bigr)^{\abs{\la}}
\frac{(e/dt,f/dt)_{\la}}{(a,q^{-N})_{\la}}\,
\P_{\nu/\la}(1)
\Q_{\mu/\la}\biggl[\frac{1-d}{1-t}\biggr]
\end{multline*}
for $adq^{1-N}=ef$.
Eliminating $f$, taking the limit $N\to\infty$ and using
\eqref{hom} this results in the claim with $(b,c)\mapsto(d,e)$.
\end{proof}

We are now prepared to prove Theorem~\ref{thmskewRC}.
\begin{proof}
We proceed by induction on $n:=\abs{X}$.
For $n=0$ we get the tautology
\[
\frac{(b/c)_{\mu}}{(b/c)_{\nu}}\,
\Q_{\mu/\nu}\biggl[\frac{a-c}{1-t}\biggr]=
\frac{(b/c)_{\mu}}{(b/c)_{\nu}}\,
\Q_{\mu/\nu}\biggl[\frac{a-c}{1-t}\biggr].
\]
For $n\geq 1$ we compute the sum on the right-hand side of 
\eqref{skewRC}, assuming the formula is true for the alphabet $Y$.
(Recall that $X=Y+z$.)
Using the branching rule \eqref{SkewR} and
\begin{equation}\label{rep}
\Q_{\la/\nu}\biggl[\frac{a-c}{1-t}\biggr] 
\P_{\la/\eta}(z) =
a^{\abs{\eta-\nu}}
\Q_{\la/\nu}\biggl[\frac{1-c/a}{1-t}\biggr] 
\P_{\la/\eta}(az)
\end{equation}
we obtain
\begin{align*}
S_{\mu,\nu}(X;a,b,c)&:=
\sum_{\la} \frac{(b/c)_{\la}}{(b/c)_{\nu}}\,
\Q_{\la/\nu}\biggl[\frac{a-c}{1-t}\biggr] \R_{\la/\mu}(X;b) \\
&\hphantom{:}=\sum_{\la,\eta} a^{\abs{\eta-\nu}}
\frac{(b/c)_{\la}(bz/t)_{\eta}}{(b/c)_{\nu}(bz)_{\la}}\,
\Q_{\la/\nu}\biggl[\frac{1-c/a}{1-t}\biggr]
\P_{\la/\eta}(az) \R_{\eta/\mu}(Y;b).
\end{align*}
The sum over $\la$ can be transformed by 
Proposition~\ref{Propskew} with $(a,b,c)\mapsto(b/c,c/a,bz)$
and $\mu\mapsto\eta$. As a result
\begin{align*}
S_{\mu,\nu}(X;a,b,c)&=
\frac{(cz,dz)_{\infty}}{(az,bz)_{\infty}} 
\sum_{\la,\eta} 
\frac{(b/c)_{\eta}(dz/t)_{\la}}{(b/c)_{\la}(dz)_{\nu}}\,
\Q_{\eta/\la}\biggl[\frac{a-c}{1-t}\biggr] 
\P_{\nu/\la}(z) \R_{\eta/\mu}(Y;b)  \\
&=\frac{(cz,dz)_{\infty}}{(az,bz)_{\infty}} 
\sum_{\la} \frac{(dz/t)_{\la}}{(dz)_{\nu}}\,
\P_{\nu/\la}(z) S_{\mu,\la}(Y,a,b,c),\
\end{align*}
where once again we have used \eqref{rep}, and where $d=ab/c$.
By the induction hypothesis, $S_{\mu,\la}(Y,a,b,c)$ may be replaced 
by the right-hand side of \eqref{skewRC} with 
$(X,\la,\nu)\mapsto(Y,\eta,\la)$, leading to
\[
S_{\mu,\nu}(X;a,b,c)=
\prod_{x\in X} \frac{(cx,dx)_{\infty}}{(ax,bx)_{\infty}} 
\sum_{\la,\eta} 
\frac{(b/c)_{\mu}(dz/t)_{\la}}{(b/c)_{\eta}(dz)_{\nu}}\,
\Q_{\mu/\eta}\biggl[\frac{a-c}{1-t}\biggr] 
\P_{\nu/\la}(z) \R_{\la/\eta}(Y;d).
\]
One final application of \eqref{SkewR} results in 
\[
S_{\mu,\nu}(X;a,b,c)=
\prod_{x\in X} \frac{(cx,dx)_{\infty}}{(ax,bx)_{\infty}} 
\sum_{\eta} 
\frac{(b/c)_{\mu}}{(b/c)_{\eta}}\,
\Q_{\mu/\eta}\biggl[\frac{a-c}{1-t}\biggr] 
\R_{\nu/\eta}(X;d),
\]
which is the right-hand side of \eqref{skewRC} with $\la\mapsto\eta$.
\end{proof}

\section{Transformation formulas for $\sn$ basic hypergeometric series}

In this final section we prove a number of additional identities for
basic hypergeometric series involving the function $\R_{\la}(X;b)$.
For easy comparison with known results for one-variable basic
hypergeometric series we define
\begin{multline*}
{_r\Phi_s}\biggl[\genfrac{}{}{0pt}{}{a_1,\dots,a_r}
{b_1,\dots,b_s};z;X\biggr] \\
:=\sum_{\la} \frac{(a_1,\dots,a_r)_{\la}}{(b_1,\dots,b_{s-1})_{\la}}\,
\Bigl((-1)^{\abs{\la}}q^{n(\la')}t^{-n(\la)}\Bigr)^{s-r+1}\,
z^{\abs{\la}} \R_{\la}(X;b_s),
\end{multline*}
where $X$ is a finite alphabet. 
There is some redundancy in the above definition since by Lemma~\ref{Lembindep}
$z^{\abs{\la}} \R_{\la}(X;b_s)=\R_{\la}(zX;b_s/z)$.
Since $\R_{\la}(X;b)$ vanishes if $l(\la)>\abs{X}$, and recalling 
$\R_{(k)}(z;b)=z^k/(bz)_k$ (see \eqref{n1}), it follows that
\begin{align*}
{_r\Phi_s}\biggl[\genfrac{}{}{0pt}{}{a_1,\dots,a_r}
{b_1,\dots,b_s};z;\{1\}\biggr]
&=\sum_{k=0}^{\infty} \frac{(a_1,\dots,a_r)_k}{(b_1,\dots,b_s)_k}\,
\Bigl((-1)^k q^{\binom{k}{2}}\Bigr)^{s-r+1}\, z^k \\
&={_r\phi_s}\biggl[\genfrac{}{}{0pt}{}{a_1,\dots,a_r}
{b_1,\dots,b_s};z\biggr],
\end{align*}
with on the right the standard notation for one-variable basic
hypergeometric series (with dependence on the base $q$ suppressed).
The reader is warned that only in this degenerate case do 
the ``lower-parameters'' $b_1,\dots,b_s$ enjoy full $\Symm_s$ symmetry.
Using the above notation the $q$-Gauss sum \eqref{newGauss} may
be stated as
\begin{equation}\label{newGauss3}
{_2\Phi_1}\biggl[\genfrac{}{}{0pt}{}{a,b}
{c};\frac{c}{ab};X\biggr]
=\prod_{x\in X} \frac{(cx/a,cx/b)_{\infty}}{(cx,cx/ab)_{\infty}}.
\end{equation}
To generalise this result we first prove the following 
transformation formula.
\begin{theorem}\label{GthmKTW}
For $f=de/bc$ and $\mu$ a partition, 
\begin{multline*}
\sum_{\la} \frac{(a,b)_{\la}}{(d)_{\la}} \,
\Q_{\la/\mu}\biggl[\frac{f/a-cf/a}{1-t}\biggr] \R_{\la}(X;e) \\
=
\frac{(b)_{\mu}}{(d/c)_{\mu}}
\biggl(\: \prod_{x\in X} \frac{(fx,ex/a)_{\infty}}{(ex,fx/a)_{\infty}} \biggr)
\sum_{\la} \frac{(a,d/c)_{\la}}{(d)_{\la}}\,
\Q_{\la/\mu}\biggl[\frac{e/a-cf/a}{1-t}\biggr] \R_{\la}(X;f).
\end{multline*}
\end{theorem}
Note that the substitution $(b,c,e,f)\mapsto
(d/c,d/b,f,e)$ interchanges the left- and right-hand sides.

\begin{proof}
We replace $\mu\mapsto\nu$ and then rename the summation 
index $\la$ on the right as $\mu$. 
By the Pieri formula of Theorem~\ref{Pieri} 
(with $(a,b,c)\mapsto (bc/d,e,a)$) the term 
\[
\R_{\mu}(X;f)
\prod_{x\in X} \frac{(fx,ex/a)_{\infty}}{(ex,fx/a)_{\infty}}
\]
on the right-hand side can be expanded as
\[
\sum_{\la}
\frac{(a)_{\la}}{(a)_{\mu}}\,
\Q_{\la/\mu}\biggl[\frac{f/a-e/a}{1-t}\biggr] \R_{\la}(X;e),
\]
resulting in
\begin{multline*}
\sum_{\la} 
\frac{(a,b)_{\la}}{(d)_{\la}} \,
\Q_{\la/\nu}\biggl[\frac{f/a-cf/a}{1-t}\biggr] \R_{\la}(X;e) \\
=\frac{(b)_{\nu}}{(d/c)_{\nu}}
\sum_{\la,\mu} 
\frac{(a)_{\la}(d/c)_{\mu}}{(d)_{\mu}}\,
\Q_{\la/\mu}\biggl[\frac{f/a-e/a}{1-t}\biggr] 
\Q_{\mu/\nu}\biggl[\frac{e/a-cf/a}{1-t}\biggr] 
\R_{\la}(X;e).
\end{multline*}
By \eqref{Saal} with $(a,b,c,d)\mapsto(d/c,b,d,e/ab)$ the sum over $\mu$
can be carried out, completing the proof.
\end{proof}

For $\mu=0$ Theorem~\ref{GthmKTW} simplifies to a multiple analogue of the
$q$-Kummer--Thomae--Whipple formula \cite[Equation (III.10)]{GR04}
(corresponding to the formula below when $X=\{1\}$).
\begin{corollary}[$\sn$ $q$-Kummer--Thomae--Whipple formula]\label{CorKTW}
For $f=de/bc$,
\begin{equation}\label{KTW}
{_3\Phi_2}\biggl[\genfrac{}{}{0pt}{}{a,b,c}
{d,e};\frac{f}{a};X\biggr]
=\biggl(\:\prod_{x\in X} \frac{(fx,ex/a)_{\infty}}{(ex,fx/a)_{\infty}}
\biggr)\:
{_3\Phi_2}\biggl[\genfrac{}{}{0pt}{}{a,d/b,d/c}
{d,f};\frac{e}{a};X\biggr].
\end{equation}
\end{corollary}
For $d=c$ this reduces to the $q$-Gauss sum \eqref{newGauss3}.

If we let $c,d\to 0$ in Theorem~\ref{GthmKTW} such that $d/c=bf/e$ and then 
replace $(e,f)\mapsto(c,az)$, we find
\begin{multline*}
\sum_{\la} z^{\abs{\la-\mu}} 
\frac{(a,b)_{\la}}{(a,b)_{\mu}} \qbin{\la}{\mu} \R_{\la}(X;c) \\
=\biggl(\:
\prod_{x\in X} \frac{(cx/a,azx)_{\infty}}{(cx,zx)_{\infty}} 
\biggr)
\sum_{\la} 
\Bigl(\frac{c}{a}\Bigr)^{\abs{\la-\mu}} 
\frac{(a,abz/c)_{\la}}{(a,abz/c)_{\mu}} \qbin{\la}{\mu} \R_{\la}(X;az).
\end{multline*}
For $\mu=0$ this is a multiple analogue of Heine's $_2\phi_1$
transformation \cite[Equation (III.2)]{GR04}.
\begin{corollary}[$\sn$ Heine transformation]
We have
\[
{_2\Phi_1}\biggl[\genfrac{}{}{0pt}{}{a,b}{c};z;X\biggr]
=\biggl(\:\prod_{x\in X} \frac{(cx/a,azx)_{\infty}}{(cx,zx)_{\infty}}
\biggr)\:
{_2\Phi_1}\biggl[\genfrac{}{}{0pt}{}{a,abz/c}{az};\frac{c}{a};X\biggr].
\]
\end{corollary}

\subsection*{Acknowledgements}
We thank Eric Rains for helpful comments on a preliminary draft of this
paper.

\bibliographystyle{amsplain}

\end{document}